\documentclass[10pt, reqno]{amsart}
\usepackage{amsfonts}
\usepackage{amssymb}
\usepackage{latexsym}
\usepackage{amsmath}
\usepackage[parfill]{parskip} 

\usepackage[]{enumerate}
\usepackage{verbatim}
\usepackage{bbm}
\usepackage{fourier}

\usepackage[utf8]{inputenc}

\usepackage[final]{hyperref}
\numberwithin{equation}{section}

\usepackage[final]{hyperref}
\hypersetup{colorlinks=true, linkcolor=blue, anchorcolor=blue, citecolor=red, filecolor=blue, menucolor=blue, pagecolor=blue, urlcolor=blue}

\newcommand{\abs}[1]{\vert #1 \vert}

\newcommand{\norm}[1]{\left\Vert #1 \right\Vert}

\newcommand{\Z}{\mathbb{Z}}
\newcommand{\R}{\mathbb{R}}

\newcommand{\angles}[1]{\langle #1 \rangle}

\DeclareMathOperator{\supp}{supp}

\newtheorem{theorem}{Theorem}
\newtheorem{proposition}{Proposition}
\newtheorem{lemma}{Lemma}
\newtheorem{corollary}{Corollary}
\newtheorem*{PWtheorem}{Paley-Wiener Theorem}

\theoremstyle{definition}

\theoremstyle{remark}
\newtheorem{remark}{Remark}

\title[Analyticity of solution for KdV]{ Asymptotic lower bound for the radius of spatial analtyicity to solutions of KdV equation}

\author{Achenef Tesfahun}
\email{achenef@gmail.com}
 
\address{Department of Mathematics\\
University of Bergen\\
PO Box 7803\\
5020 Bergen\\ Norway}



\begin{document}

\begin{abstract}
  It is shown that the uniform radius of spatial analyticity $\sigma(t)$ of solutions at time $t$ to the KdV equation cannot decay faster than 
$|t|^{-4/3}$ as $|t| \to \infty$ given initial data that is analytic with fixed radius $\sigma_0$. This improves a recent result of Selberg and Da Silva, where they proved a decay rate of $|t|^{-(4/3 + \varepsilon) }$ for arbitrarily small positive $\varepsilon$. The main ingredients in the proof are almost conservation law for the solution to the KdV equation in space of analytic functions and space-time  dyadic bilinear $L^2$ estimates associated with the KdV equation.
\end{abstract}

\maketitle

\section{Introduction}

Consider the Cauchy problem for KdV equation 
\begin{equation}\label{kdv}
\left\{
\begin{aligned}
  & u_t + u_{xxx} + u u_x = 0,
  \\
  &u(0,x) = f(x),
\end{aligned}
\right.
\end{equation}
where the unknown is 
$$
  u(t,x) \colon \R \times \mathbb R \to \mathbb R.
$$

This equation was derived by Korteweg and de Vries \cite{KdV1895} as a model for long wave propagating in a channel. The well-posedness theory of \eqref{kdv} has been extensively studied, for instance, Kenig, Ponce and Vega 
 \cite{KPV1996} proved local well-posedness in $H^s$ for $s > -3/4$. Later, this was extended to a global result by Colliander, Keel, Staffilani, Takaoka and Tao \cite{CKSTT2003}.  Moreover, Christ, Colliander and Tao \cite{CCT2003} proved that the solution map of \eqref{kdv} fails to be uniformly continuous in $H^s$ for $s<-3/4$ which was first proved by Kenig, Ponce, and Vega \cite{KPV2001} for the complex-valued problem.
 More recently,  Guo \cite{Guo2009} established a global well-posedness result in $H^{-3/4}$ which is sharp in the sense of  \cite{CCT2003}. 

In this work, we are interested in the persistence of spatial analyticity for the solutions of \eqref{kdv}, given initial data in a class of
analytic functions. This is motivated naturally by observing that many special solutions of \eqref{kdv} 
such as for instance solitary and cnoidal waves are analytic in a strip about the real axis. 
For 
 real-analytic initial data $f$ with 
 uniform radius of analyticity $\sigma_0>0$, so there is a holomorphic extension to a complex strip 
 $$ S_{\sigma_0} =\{x + iy : |y| < \sigma_0 \},$$
it was established in \cite{GK2002} that for small $t$
 the solution $u$ of \eqref{kdv} is analytic in $S_{\sigma(t)}$ with $\sigma(t)=\sigma_0$, i.e., 
 the radius of analyticity remains constant for short times. 
For large times on the other hand it was shown in \cite{BGK2005} that 
$\sigma(t)$ can decay no faster than 
$|t|^{-12}$ as $t \to \infty$. This is improved  greatly more  recently by Selberg and Da Silva \cite{SD2015} to a decay rate of $|t|^{-(4/3+\varepsilon)}$, where $0<\varepsilon\ll 1$ is sufficiently small.
In the present paper we are able to remove the $\varepsilon$ exponent, and thus improving the decay rate further to $|t|^{-4/3}$. The exponent $-4/3$ turn out to be related to the Sobolev regularity exponent to $H^{-3/4}$ (specfically, one is the reciprocal of the other) at which Guo \cite{Guo2009} obtained a sharp well-posedness result. The main ingredients in our proof are almost conservation law for the solution to the KdV equation in spaces of analytic functions and space-time dyadic bilinear estimates associated with the KdV equation. For similar studies for the Dirac-Klein-Gordon system, generalized KdV  and cubic NLS see \cite{ST2015, HKS2017, ST2017,  AT2017}. 
 For studies on related issues for nonlinear partial differential equations see for instance \cite{  DHK1995, Ferrari1998, GGYTE2015,  GHHP2013, HHP2011,  HP2012, Himonas2009, H1990,  Kappeler2006, Oliver2001,   KM1986, Panizzi2012, Levermore1997}.

A class of analytic function spaces suitable to study analyticity of solution is the analytic Gevrey class. These spaces are denoted $G^{\sigma,s} = G^{\sigma,s}(\mathbb R)$ with a norm given by
\[
\| f \|_{G^{\sigma, s}} = \left\| e^{\sigma | D_x |}\angles{ D_x }^s f\right\|_{L^2_x},
\]
where $D_x = -i\partial_x$ with Fourier symbol $\xi$ 
and $\angles{\cdot}=\sqrt{1+|\cdot|^2}$.
We write 
 $$
 G^{\sigma} := G^{\sigma,0}.
 $$
 For $\sigma=0$ the Gevrey-space coincides with the Sobolev space $H^s$.

One of the key properties of the Gevrey space is that every function in $G^{\sigma,s}$ with $\sigma>0$ has an analytic extension to the strip $S_\sigma$. This property 
is contained in the following Theorem which is proved in
 \cite[p. 209]{K1976} for $s = 0$; the argument applies also for $s\in \R$ with some obvious
modifications..
\begin{PWtheorem} 
Let $\sigma > 0$ and $s\in \R$.  Then the following are equivalent:
\begin{enumerate}[(i)]
\item $f \in G^{\sigma, s}$.
\item $f$ is the restriction to the real line of a function $F$ which is holomorphic in the strip
\[
S_\sigma =  \{ x + i y :\ x,y \in \mathbb{R}, \  | y | < \sigma\}
\]
and satisfies
\[
\sup_{| y | < \sigma} \| F( x + i y ) \|_{L^2_x} < \infty.
\]
\end{enumerate}
\end{PWtheorem}

Observe that the Gevrey spaces satisfy the following embedding property:
\begin{align}
\label{Gembedding}
  G^{\sigma, s} &\subset G^{\sigma', s'} \quad \text{for all $ 0\le \sigma'< \sigma$ and $s, s'\in \mathbb R$}.
\end{align}
As a consequence of this property and the existing well-posedness theory in $H^s$ we conclude that the Cauchy problem \eqref{kdv} has a unique, 
smooth solution for all time, given initial data 
$f  \in G^{\sigma_0}$ for all $\sigma_0>0$.
Our main result gives an algebraic lower bound on the radius of analyticity $\sigma(t)$ of the solution as the time $t$ tends to infinity.
\begin{theorem}\label{thm-gwp}
Assume $f \in G^{\sigma_0}$ for some $\sigma_0 > 0$. Let $u$ be the global $C^\infty$--solution of \eqref{kdv}. Then $u$ satisfies
\[
  u(t) \in G^{\sigma(t)} \quad \text{for all } \ t\in \R
\]
with the radius of analyticity $\sigma(t)$ satisfying an asymptotic lower bound
\[
\sigma(t) \ge c|t|^{-\frac 43} \quad \text {as} \ |t|\rightarrow \infty,
\]
where $c > 0$ is a constant depending on  $\|f\|_{G^{\sigma_0} }$ and $\sigma_0$.
\end{theorem}

By time reversal symmetry of \eqref{kdv} we may from now on restrict ourselves to positive times $t\ge 0$. 
The first step in the proof of Theorem \ref{thm-gwp} is to show that in a short time interval $0 \le  t \le  t_0$, where $t_0> 0$ depends on the norm of the initial data, 
the radius of analyticity remains strictly positive. This is proved using a standard contraction argument involving energy type estimates, 
and a bilinear estimate in Bourgain-Gevrey type space; the proofs are given in section 4.
The next step is to improve the control on the growth of the solution in the time interval $[0, t_0]$, measured in the data norm $G^{\sigma_0}$.
To achieve this we show that, although the conservation of  $G^{\sigma_0}$-norm of solution does not hold exactly, it does hold in an approximate sense (see Section 5.1).
This approximate conservation law will allow us to iterate the local result and obtain the asymptotic lower bound on $\sigma $ in Theorem \ref{thm-gwp} (see Section 5.2).


\section{Preliminaries, Functions spaces and linear estimates}


\subsection{Preliminaries}
First we fix notation.
In equations, estimates and summations capitalized variables such as $N$ and $L$ are presumed to be dyadic with $N , L >  0$, i.e., these
variables range over numbers of the form $ 2^k$ for $k \in \Z$.
In estimates we use $A\lesssim B$ as shorthand for $A \le CB$ and $A\ll B$ for $A\le C^{-1} B$, where $C\gg1 $ is a positive constant which is independent of dyadic numbers such as $N$ and $L$; $A\sim B$ means $B\lesssim A\lesssim B$; 
$\mathbb{1}_{\{\cdot\}}$  denotes the indicator function which is 1 if the condition in the bracket is satisfied and 0 otherwise; we write  $a\pm:=a\pm \varepsilon$ for sufficiently small $0<\varepsilon\ll 1$. Finally, we use the notation
$$\| \cdot \|=\| \cdot \|_{L^2_{t,x}(\R^{1+1})}.$$

Consider an even function $ \chi \in C_0^\infty((-2, 2))$ such that 
$\chi (s)=1 $ if $|s|\le 1$.
 Define
\begin{align*}
\beta_{N}(s)&=\begin{cases}
& 0, \quad \text{if} \ N<1,
\\
& \chi(s), \quad \text{if} \ N=1,
\\
&\chi\left(\frac{s}{N}\right)-\chi \left(\frac{2s}{N}\right),\quad \text{if} \ N>1.
\end{cases}
 \end{align*}
 Thus
 \begin{align*}
\supp \beta_1 \subset \left\{s\in \R: |s|\le 2 \right \}, \quad \supp \beta_N \subset \left \{s\in \R: \frac N 2\le  |s|\le 2N\right\}  \  \ \text{for} \ N>1.
 \end{align*}
Note that 
\begin{equation}\label{betasum}
\sum_{N\ge 1}\beta_{N}(s)=1 \quad \text{for} \ s\neq 0.
\end{equation}

The Fourier transform in space and space-time are given by
\begin{align*}
\mathcal F_{x} (f)(\xi)=&\widehat f(\xi)=\int_{\R} e^{ -i x  \xi} f(x) \ dx,
\\
\mathcal F_{t, x} (u)(\tau,\xi)&=\widetilde u(\tau, \xi)=\int_{\R^{1+1}} e^{ -i(t\tau+ x  \xi)} u(t,x) \ dt dx.
\end{align*}
Now define 
\begin{align*}
P_{N} u&= \mathcal F_x^{-1}\left[ \beta_N(\xi)\widehat{ u})\right] \ \ \text{for} \ N\ge 1,
\\
Q_{L} u&= \mathcal F_{t,x}^{-1}\left[\beta_L\left( \tau- \xi^3\right)\widetilde{ u} \right]  \ \ \text{for} \ L\ge 1.
 \end{align*}
 Here $N$ and $L$ measure the magnitude of the spatial frequency and modulation, respectively. We use the notation 
$$u_N:=P_N u, \quad u_{N, L}:=P_N Q_L u.$$ 
In view of \eqref{betasum} one can write
 \begin{align*}
u=\sum_{N\ge 1 } u_{N} .
 \end{align*}

In addition to $P_N$ and $Q_N$ we also need the homogeneous projections $\dot P_N$ and $\dot Q_L$ defined by
\begin{align*}
\dot P _Nu &=\mathcal F_{x}^{-1}\left[\mathbbm{1}_{ \left\{\frac N2 \le |\xi|\le 2 N\right\}} \widehat u\right] \ \ \text{for} \ N> 0,\\
\dot Q_Lu &= \mathcal F_{t,x}^{-1}\left[ \mathbbm{1}_{ \left\{ \frac L2 \le |\tau-\xi^3|\le 2 L \right\}} \widetilde u\right] \ \ \text{for} \ L>0.
\end{align*}
Note that
\begin{equation}\label{PP}
P_Nu =\dot P_N P_N u, \quad Q_Lu =\dot Q_L Q_L u \quad  \text{for} \ \ N, L>1 .
\end{equation}

\begin{remark}\label{rmk-dyadicsum} We shall make a frequent use of of the following dyadic summation estimate:
 For $N \in 2^{\Z}$, $1\le \alpha <\beta$ and $a\in \R$ we have
\begin{equation}\label{dyadicsum}
\sum_{\alpha\le N \le \beta} N^a \sim 
\begin{cases}
&\beta^a \quad  \text{if} \ \ a>0,
\\
&\log(\beta/\alpha) \quad  \text{if} \ \ a=0,
\\
&\alpha^a \quad  \text{if} \ \ a<0.
\end{cases}
\end{equation} 
 
 \end{remark}
 

\subsection{Function spaces}
For $1\le q, r \le \infty$ the mixed space-time Lebesgue space $L_t^q L_x^r(\R^{1+1})$ is defined with the norm
$$
\|u\|_{L_t^q L_x^r}= \| \|u(t, \cdot)\|_{L_x^r}\|_{L_t^q }=\left(\int_{\R} \left(\int_{\R} |u(t,x)|^r \, dx\right)^{\frac qr} \, dt\right)^\frac 1q
$$
with an obvious modification when $q=\infty$ or $r=\infty$, and when the space is restricted to bounded intervals. Similarly 
$$
\|u\|_{L_x^r L_t^q}= \| \|u(\cdot, x)\|_{L_t^q}\|_{L_x^r}.
$$
 For an interval $I$ we write $L_x^r  L_I^q$ to denote $L_x^r  L_t^q (I \times \R)$, i.e., the time variable is restricted to $I$.

The Bourgain space associated with the KdV equation, denoted  $X^{s,b}$, is defined as the completion of the Schwartz class $\mathcal S(\R^{1+1})$ with respect to the norm
\begin{align*}
\norm{u}_{X^{s, b}}&=\left(\sum_{N, L \geq
1}N ^{2s} L^{2b} \norm{ u_{N,L} }^2 \right)^\frac12.
\end{align*}
The restriction to a time slab $I \times \mathbb R $ of $ X^{s,b}$, denoted $ X^{s,b}_I$, is a Banach space when equipped with the norm
$$
\| u \|_{ X^{s,b}_I}  = \inf \left\{ \| v \|_{X^{s,b}}: \ v = u \text{ on } I \times \mathbb{R} \right\}.
$$
By a standard contraction argument in the $ X^{s,b}_I$--space local well-posedness of \eqref{kdv} for $H^s$ data reduces to the bilinear estimate
\begin{equation}\label{biest1}
\norm{\partial_x(uv)}_{X^{s, b-1}}\lesssim \norm{u}_{X^{s, b}}\norm{v}_{X^{s, b}}
\end{equation}
for some $b>1/2$.

 In \cite{KPV1996} Kenig, Ponce and Vega proved that \eqref{biest1} holds for $s>-3/4$, but fails for $s<-3/4$.
Later, it was also shown by Nakanishi, Takaoka and Tsutsumi \cite{NTT2001} that
 \eqref{biest1} also fails to hold at the borderline $s=-3/4$. A usual approach to resolve problems such as this is to modify the Bourgain space by setting $b=\frac12$ and replacing the $l^2$-summation in the modulation parameter, $L$, by $l^1$-summation.
This space which we denote by $X^s$ is defined with respect to the norm
\begin{align*}
\norm{u}_{X^s}&=\left(\sum_{N \geq
1}N ^{2s} \norm{ u_N}^2_{X} \right)^\frac12,
\end{align*}
where 
$$
\norm{v}_{X}=\sum_{L \geq
1}L^\frac12 \norm{Q_L v}.
$$
Note that if 
$u_N \in X$ then
\begin{equation}\label{X-timeloc}
\norm{\gamma(M(t-t_0)) u_N}_{X} \lesssim \norm{ u_N}_{X} 
\end{equation}
for all $M, N\ge 1$, $t_0\in \R$ and $\gamma \in \mathcal S(\R)$.   
Indeed, by definition 
\begin{align*}
\norm{\gamma(M(t-t_0)) u_N}_{X}&=\sum_{L \geq
1}L^\frac12 \norm{ \gamma(M(t-t_0)) Q_L u_N} 
\\
&\le \| \gamma(M(\cdot-t_0)) \|_{L_t^\infty} \sum_{L \geq
1}L^\frac12 \norm{  Q_L u_N} 
\\
&\lesssim \norm{ u_N}_{X}.
\end{align*}

The restriction to a time slab $I \times \mathbb R $ of $ X^{s}$, denoted $ X^{s}_I$, is defined similarly as above.
Now using $X^s_I$ as a contraction space local well-posedness in $H^{-3/4}$ will follow if one proves the bilinear estimate 
\begin{equation}\label{biest11}
\norm{ \mathcal B (u, v)}_{ X^{-\frac34}} 
\lesssim \norm{ u}_{X^{-\frac34 }} \norm{ v}_{ X^{-\frac34}},
\end{equation}
where
$$
\mathcal B(u, v)(t)=\chi(t/4) \int_0^t S(t-t') \partial_x \left( ( \chi u \cdot  \chi v )(t') \right) \, dt'
$$
is the time localized Duhamel term associated to the KdV equation.

However, as pointed out in \cite{Guo2009}  in trying to establish the bilinear estimate \eqref{biest11} a particular case of high:high-low frequency interaction introduces a logarithmic  derivative loss, and thus 
\eqref{biest11} is an open problem. To resolve this problem a version of $ X^s $ that is modified with respect to low frequency modes (corresponding to $N=1)$ is introduced. The new space, denoted \footnote{In \cite{Guo2009} the spaces $X^s $ and $\bar X^s $ are denoted as $F^s $ and  $\bar F^s $, respectively. } here $\bar X^s $, is defined with respect to the norm
$$
\norm{u}_{ \bar X^s}=\left( \norm{ u_1}^2_{L_x^2 L^\infty_t}+\sum_{N >
1}N ^{2s} \norm{ u_N}^2_{X} \right)^\frac12 ,
$$
where the additional $L_x^2 L^\infty_t$-norm for the low frequency helps to avoid the logarithmic divergence in the bilinear estimate \eqref{biest11}. 
The restriction to a time slab $I \times \mathbb R $ of $ \bar X^{s}$, denoted $ \bar X^{s}_I$, is defined similarly as before.
By using this space Guo \cite{Guo2009} proved the bilinear estimate
\begin{equation}\label{biest2}
\norm{ \mathcal B (u, v)}_{\bar X^{-\frac34}} 
\lesssim \norm{ u}_{\bar X^{-\frac34 }} \norm{ v}_{\bar X^{-\frac34}}
\end{equation}
thereby establishing an endpoint local well-posedness result for \eqref{kdv}.


\subsection{Linear estimates}

Let $S(t)f=e^{-t\partial_x^3} f$ be the solution to the Airy equation (free solution to the KdV equation). The following Lemma contains frequency localized Strichartz estimates, maximal function estimates and smoothing effect estimates for solution to the Airy equation (see e.g. \cite{Guo2008, KPV1991, KPV1991b}).  By the transfer principle (see e.g. \cite [Lemma 3.2]{Guo2009}) these estimates can be extended to hold for any function in $X$.
\begin{lemma} \label{lm-strz} 
 Let $I$ be a interval with $|I|\le 1$, $N\ge 1$ and  $M> 1$ be dyadic numbers. Let the pair $(q,r)$ satisfies
 $$
 2\le q, r \le \infty, \quad \frac3q+\frac1r=\frac12.
 $$
  
\begin{enumerate} [(a)]
\item For all  $f\in \mathcal{S}(\R)$ we have the following:
 \begin{align}
\label{str1}
\norm{  S(t)  f_N}_{L_t^q L_x^r} &\lesssim \norm{ f_N }_{L^2_x},
\\
\label{str2}
\norm{S(t) f_N }_{L_x^2  L_I^\infty} &\lesssim N^\frac34 \norm{  f_N }_{L^2_x},
\\
\label{str3}
\norm{S(t)  f_N }_{L_x^4  L_t^\infty} &\lesssim N^\frac14
 \norm{  f_N }_{L^2_x},
\\
\label{str4}
\norm{S(t)  f_M  }_{L_x^\infty L_t^2} &\lesssim M^{-1} \norm{ f_M }_{L^2_x}.
\end{align}

\item For all $u_N \in X$, we have
\begin{align}
\label{stre1}
\norm{u_N }_{L_t^q L_x^r} &\lesssim \norm{u_N }_{X},
\\
\label{stre2}
\norm{u_N  }_{L_x^2  L_I^\infty} &\lesssim N^\frac34 \norm{u_N }_{X},
\\
\label{stre3}
\norm{u_N }_{L_x^4  L_t^\infty} &\lesssim  N^\frac14\norm{u_N }_{X},
\\
\label{stre4}
\norm{u_M }_{L_x^\infty L_t^2} &\lesssim M^{-1} \norm{u_M }_{X} .
\end{align}

\end{enumerate}

\end{lemma}

We also have the following embedding estimates.
\begin{lemma}\label{lm-embed}
\leavevmode
\begin{enumerate} [(i)]
\item \label{emb1}
 Let $1\le N\lesssim 1 $. For all $s\in \R$ and $u\in \bar X^{s} $ we have 
\begin{equation*}
\norm{ u_{N}}_{L_t^\infty L_x^2 }  \lesssim  \norm{u}_{ \bar X^s}.
\end{equation*}
\item \label{emb2} For all $u_1 \in X$ we have
\begin{equation*}
\norm{ u_1}_{L_x^2 L^\infty_I} \lesssim \norm{  u_1}_{ X}.
\end{equation*}
\item \label{emb3}
 For all $s\in \R$ we have $\bar X^{s} \subset C(\mathbb R,H^{s})$
and
\begin{equation*}
\sup_{t\in \R}\norm{u(t) }_{ H^s} \lesssim  \norm{u}_{ \bar X^s}.
\end{equation*}

\item \label{emb4} 
For all $s_1\le s_2$ we have $\bar X^{s_2} \subset \bar X^{s_1}$.

\end{enumerate}
\end{lemma}
\begin{proof}
First we prove \eqref{emb1}. For  $1\le N\lesssim 1 $ we have by \eqref{stre1} with $(q,r)=(\infty, 2)$
$$
\norm{ u_{N}}_{L_t^\infty L_x^2 }  \lesssim  N^s \norm{u_N}_{ X} \quad \text{for all } \ \ s. 
$$
Combining this with the definition of $\bar X^s$ and the simple estimate
$$
\norm{ u_1 }_{L_t^\infty L_x^2 } \lesssim \norm{ u_1 }_{L_x^2 L^\infty_t }
$$
we obtain \eqref{emb1}.

 The inequality \eqref{emb2} follows from
\eqref{stre2} with $N=1$ whereas  
 \eqref{emb3} follows from the definition of $\bar X^{s}$ and \eqref{stre1}, i.e.,
\begin{align*}
\norm{u }^2_{L_t^\infty  H^s} &\sim  \norm{u_1}^2_{L_t^\infty L_x^2} + \sum_{N>1} N^{2s}\norm{u_N}^2_{ L_t^\infty L^2_x}
\\
&\lesssim   \norm{u_1}^2_{L_x^2 L_t^\infty } + \sum_{N>1} N^{2s}\norm{u_N}^2_{X}= \norm{u}^2_{ \bar X^s},
\end{align*}

Finally, \eqref{emb4} is simple to prove.
\end{proof}

Define the operator $\Lambda$ by
$$
\Lambda u=\mathcal F^{-1}_{t, x}\left[ \left(i+\tau-\xi^3\right) \widetilde u\right].
$$
We remark that since $|i+\tau-\xi^3|=\angles{\tau-\xi^3}$ the operator $\Lambda^{-1} $ is not singular.

\begin{lemma}[Energy type estimates] \label{lm-linearest}
\leavevmode

\begin{enumerate} [(a)]
\item \label{linearest1} Assume $f\in H^s$ for $s\in \R$. Then there exists a constant $C>0$ such that
\begin{equation}\label{linest}
\norm{\chi(t) S(t) f}_{\bar X^s} \le C \norm{ f}_{H^s},
\end{equation}
\item  \label{linearest2}  Assume $N\ge 1$ and $\Lambda^{-1} F_N \in X$ . Then there exists a constant $C>0$ such that
\begin{equation}\label{duhmest}
\norm{\chi(t) \int_0^t S(t-t') F_N(t') \, dt'}_{X} 
\le C \norm{ \Lambda^{-1} F_N}_{X}.
\end{equation}
\end{enumerate}
\begin{proof}
Part \eqref{linearest1} follows from the definition of $\bar X^s$ and Lemma \ref{lm-strz}, \eqref{str2}. Indeed, using \eqref{str2} we obtain
\begin{equation}\label{energy1}
\norm{\chi(t) S(t) f_1}_{L^2_xL_t^\infty} \lesssim \norm{ f_1}_{L^2_x}.
\end{equation}
On the other hand, we have
\begin{align*}
\mathcal F_{t,x}\left [Q_L \left( \chi(t) S(t) f_N \right) \right](\tau, \xi)= \beta_L (\tau-\xi^3) \widehat{\chi}(  \tau-\xi^3)  \widehat{f_N}(\xi),
\end{align*}
where we used the fact that $\widehat{S(t)f}(\xi)=e^{it\xi^3} \widehat f(\xi)$. 
Then by Plancherel 
\begin{equation}\label{freeL2est}
\begin{split}
\norm{Q_L( \chi(t) S(t) f_N )} &=\norm{\beta_L (\tau-\xi^3) \widehat{\chi}(  \tau-\xi^3)  \widehat{f_N}(\xi) }
\\
&=\norm{P_L \chi}_{L_t^2} \norm{f_N}_{L_x^2} .
\end{split}
\end{equation}
This in turn implies
\begin{equation}\label{energy2}
\begin{split}
\norm{\chi(t) S(t) f_N }_X&=\sum_{L\ge 1} L^\frac12 \norm{P_L \chi}_{L_t^2} \norm{f_N}_{L_x^2} 
\\
&\lesssim \norm{ \chi}_{H_t^1}  \norm{f_N}_{L_x^2}
 \lesssim  \norm{f_N}_{L_x^2} ,
\end{split}
\end{equation}
where to obtain the second inequality we used Cauchy-Schwarz, i.e., 
\begin{equation}\label{CS}
\sum_{L\ge 1} L^\frac12 \norm{P_L \chi}_{L_t^2} \le
 \left(\sum_{L\ge 1} L^{-1} \right)^\frac12  \left(\sum_{L\ge 1} L^{2} \norm{P_L \chi}^2_{L_t^2} \right)^\frac12  
\\
\lesssim \norm{ \chi}_{H_t^1} .
\end{equation}
Now using the estimates \eqref{energy1} and \eqref{energy2} in
the definition of $\bar X^s$ we obtain \eqref{linearest1}.

Variants of part \eqref{linearest2} has appeared in the literature, see for instance \cite{Guo2009a}. For completeness we give the proof here by adapting the proof of ( \cite[Section 13.1]{DS2011}).
To this end we let
$$
u_N(t)= \int_0^t S(t-t') F_N(t') \, dt'.
$$

Taking
Fourier transform in space,
\begin{align*}
  \widehat u_N(t,\xi)
  &=
  \int_0^t e^{ i(t-t')\xi^3} \widehat F_N(t',\xi) \, dt'
  \simeq
  \int \frac{e^{it\lambda}-e^{it\xi^3}}{i(\lambda-\xi^3)} \widetilde F_N(\lambda,\xi) \, d\lambda
  \\
    &=\left(
  \int_{\left\{ \abs{\lambda -\xi^3} \lesssim 1\right\} } +  \int_{\left\{ \abs{\lambda -\xi^3} \gg 1\right\} } \right) \frac{e^{it\lambda}-e^{it\xi^3}}{i(\lambda-\xi^3)} \widetilde F_N(\lambda,\xi) \, d\lambda
  \\
  &:=\widehat v_N(t,\xi)+ \widehat w_N(t,\xi).
\end{align*}

\subsection*{Estimate for $v_N$}
Expanding we write
\begin{align*}
  \widehat{v_N}(t,\xi)
  =
  e^{ it\xi^3} \sum_{k=1}^\infty \int_{\left\{\abs{\lambda -\xi^3} \lesssim 1\right\}} \frac{\left[it(\lambda -\xi^3)\right]^k}{k! i(\lambda - \xi^3)} \widetilde F_N(\lambda,\xi) \, d\lambda
\end{align*}
and hence
\begin{equation}\label{vN}
  v_N(t) =
  \sum_{k=1}^\infty \frac{t^k}{k!}  S(t) g_k,
\end{equation}
where
\begin{align*}
  \widehat{g_k}(\xi) = \int_{\left\{ \abs{\lambda -\xi^3} \lesssim 1\right\} }  \left[i(\lambda -\xi^3)\right]^{k-1}  \widetilde F_N(\lambda,\xi) \, d\lambda.
\end{align*}
Set $\psi_k(t)=t^k \chi(t)$. In view of \eqref{vN} and \eqref{CS} we have
\begin{align*}
  \norm{\chi(t) v_N}_{X}
  = \sum_{L\ge 1} L^\frac12 \norm{Q_L( \chi(t) v_N )}
  &\le \sum_{k=1}^\infty \frac{1}{k!}  \sum_{L\ge 1} L^\frac12 \norm{P_L  \psi_k }_{L_t^2} \norm{g_k}_{L_x^2} 
\\
&\le \sum_{k=1}^\infty \frac{1}{k!}   \norm{  \psi_k }_{H_t^1} \norm{g_k}_{L_x^2} .
\end{align*}
But $ \norm{  \psi_k }_{H_t^1} \lesssim 2^k + k 2^{k-1}$,  and by Cauchy-Schwarz
\begin{align*}
  |\widehat{g_k}(\xi) |^2&= \left( \int_{|\lambda -\xi^3| \lesssim 1}  |\lambda -\xi^3 |^{2(k-1)}| |i+\lambda -\xi^3|^2 \, d\lambda \right) 
  \\   
   &  \qquad \qquad \cdot \left( \int_{| \lambda -\xi^3| \lesssim 1}  | i+\lambda -\xi^3|^{-2} |\widetilde F_N(\lambda,\xi) |^2\, d\lambda \right)
  \\
 & \lesssim \int_{| \lambda -\xi^3| \lesssim 1}  | i+\lambda -\xi^3|^{-2} |\widetilde F_N(\lambda,\xi) |^2\, d\lambda.
\end{align*}
By Plancherel we have
$
\norm{g_k}_{L_x^2}\lesssim \norm{\Lambda^{-1} F_N}$, and hence 
\begin{align*}
  \norm{\chi(t) v_N}_{X}
&\lesssim \norm{\Lambda^{-1} F_N}  \left(\sum_{k=1}^\infty \frac{ 2^k + k 2^{k-1}}{k!} \right)
\\
& \lesssim  \norm{\Lambda^{-1} F_N}  \lesssim  \norm{\Lambda^{-1} F_N}_X  ,
\end{align*}
where in the last inequality we used 
\begin{align*}
    \norm{\Lambda^{-1} F_N}^2  \sim    \sum_{L\ge 1}\norm{\Lambda^{-1} Q_L F_N}^2  \lesssim
    \norm{\Lambda^{-1} F_N}^2_X  .
\end{align*}

\subsection*{Estimate for $w_N$}
Taking Fourier transform in time
\begin{align*}
  \widetilde w_N(\tau,\xi)
  &=
  \int_{ \left\{ \abs{\lambda -\xi^3} \gg 1\right\}} \frac{\delta(\tau-\lambda)-\delta(\tau- \xi^3)}{i(\lambda- \xi^3 )}  \widetilde F_N(\lambda,\xi) \,
  d\lambda\\
 & =
  \frac{ \mathbbm{1}_{  \left\{ \abs{\lambda -\xi^3} \gg 1\right\}
  }\widetilde F_N(\tau,\xi)}{i(\tau- \xi^3)}
  - \delta(\tau- \xi^3) \int \frac{ \mathbbm{1}_{  \left\{ \abs{\lambda -\xi^3} \gg 1\right\}
  }\widetilde F_N(\lambda,\xi)}{i(\lambda - \xi^3 )}  \, d\lambda
  \\
  &:= \widetilde y_N(\tau,\xi) -  \widetilde z_N(\tau,\xi).
\end{align*}

Obviously, (see \eqref{X-timeloc}) we can estimate $y_N$ as
\begin{align*}
  \norm{\chi(t)y_N}_{X}
  \lesssim \norm{y_N}_{X}&=\sum_{L\ge 1} L^\frac12 \norm{Q_L y_N }
  \\
  &\lesssim  \sum_{L\ge 1} L^\frac12 \norm{
   Q_L (\Lambda^{-1}F_N) } =\norm{\Lambda^{-1}F_N}_{X}.
\end{align*}

On the other hand, write
$ \widetilde z_N(\tau,\xi)=\delta(\tau- \xi^3) \widehat h_N(\xi), $
where
$$
  \widehat h_N(\xi)
  =
  \int \frac{  \mathbbm{1}_{  \left\{ \abs{\lambda -\xi^3} \gg 1\right\}} \widetilde F_N(\lambda,\xi) }{i(\lambda-\xi^3)}  \, d\lambda.
$$
By Plancherel
\begin{align*}
 \norm{Q_L (\chi (t)z_N )} = \norm{ \beta_L(\tau-\xi^3) \widehat{\chi}(\tau-\xi^3) \widehat h_N (\xi)} =\norm{ P_L\chi}_{L_t^2} \norm{h_N}_{L_x^2}
\end{align*}
But by dyadic decomposition and Cauchy-Schwarz
\begin{align*}
  |\widehat{h_N}(\xi) |
  \lesssim \int   \mathbbm{1}_{  \left\{ \abs{\lambda -\xi^3} \gg 1\right\}}   | \widetilde { ( \Lambda^{-1}F_N) }(\lambda,\xi) |\, d\lambda 
 &\lesssim \sum_{L\gg 1   } \int  \mathbbm{1}_{  \{ \abs{\lambda -\xi^3}\sim L \} } |\widetilde { ( \Lambda^{-1}F_N) }(\lambda,\xi) |\, d\lambda 
 \\
& \le  \sum_{L\gg 1   }  L^\frac12  \left( \int \mathbbm{1}_{  \{ \abs{\tau -\xi^3}\sim L \} } |\widetilde { ( \Lambda^{-1}F_N) }(\tau,\xi) |^2 \, d\tau \right)
\end{align*}
and hence
\begin{align*}
\norm{h_N}_{L_x^2}
  \lesssim  \sum_{L\ge 1} L^\frac12 \norm{ Q_L (\Lambda^{-1}F_N) }=\norm{\Lambda^{-1}F_N}_{X}.
\end{align*}
Therefore,
\begin{align*}
\norm{ \chi(t) z_N}_{X}&= 
\sum_{L\ge 1} L^\frac12 \norm{Q_L (\chi (t)z_N )}
\\
&=  
\sum_{L\ge 1} L^\frac12 \norm{ \beta_L(\tau-\xi^3) \widehat{\chi}(\tau-\xi^3) \widehat h_N (\xi)}
\\
& =  
\sum_{L\ge 1} L^\frac12  \norm{ P_L\chi}_{L_t^2} \norm{h_N}_{L_x^2}
&
\\
&\lesssim \norm{\chi}_{H_t^1}\norm{h_N}_{L_x^2}
  \lesssim \norm{\Lambda^{-1}F_N}_{X}.
\end{align*}

 \end{proof}

\end{lemma}


\section{Bilinear estimates}

For dyadic numbers $N_j>0$ ($j=1,2,3$) we denote by 
$ N_{\text{min}}$, $ N_{\text{med}}$ and $ N_{\text{max}}$ the minimum, median and maximum of $(N_1, N_2, N_3)$. We use similar notation for  $L_j>0$ ($j=1,2,3$). 

Following the methods in \cite{Tao2001} the bilinear estimate in $X^{s,b}$-space that is needed to obtain local well-posedness of \eqref{kdv} reduces to establishing dyadic bilinear estimates of the form
\begin{equation} \label{BiNLest}
\norm{\dot P_{N_3}\dot Q_{L_3} 
\left(  \left( \dot P_{N_1} \dot Q_{L_1} u_1 \right) \left(  \dot P_{N_2} \dot Q_{L_2} u_2\right) \right)} \le C(N,L) \prod_{j=1}^2 \|   \dot P_{N_j} \dot Q_{L_j} u_j \|
\end{equation}
for some  \footnote{ In \cite{Tao2001} the optimal constant is denoted by $\norm{m}_{ [3; \R \times \R]}, $ where $$m=m(\tau_j, \xi_j)=\prod_{j=1}^3  \mathbbm{1}_{ \left\{ |\xi|\sim N_j \right\} } \mathbbm{1}_{ \left\{ |\tau-\xi^3|\sim  L_j \right\} }. $$} optimal constant $C(N,L) $ that is a function of $N_j, L_j >0$ 
 ($j=1,2,3$).

By checking the support properties in Fourier space of the bilinear term on the left hand side of \eqref{BiNLest} one can see that this term vanishes unless the following conditions are satisfied (see (29) and (30) in \cite{Tao2001}): 
\begin{align}
\label{Ncompare}
N_{\text{max}}& \sim N_{\text{med}},
\\
\label{lmaxmed}
L_{\text{max}} &\sim  \max (N_{\text{min}}  N^{2}_{\text{max}}, L_{\text{med}} ).
\end{align}
We may thus assume \eqref{Ncompare} and \eqref{lmaxmed} throughout the paper.

\begin{proposition}[\cite{Tao2001}, Proposition 6.1]\label{prop-dydbiest}
Let $N_j, L_j>0$ ($j=1,2,3$) be dyadic numbers.
 Then \eqref{BiNLest} holds 
with $C(N,L)$ as follows:
\begin{enumerate} [(a)]
\item \label{cnl1} If $N_{\text{max}} \sim N_{\text{min}}$ and $L_{\text{max}}\sim N_{\text{min}}  N^{2}_{\text{max}}$, then 
$$
C(N,L)\sim  N^{-\frac{ 1} 4}_{\text{max}}  L^\frac12_{\text{min}} L^\frac14_{\text{med}} .
$$
\item  \label{cnl2} If $N_2\sim N_3 \gg N_1  $ and $N_{\text{min}}  N^{2}_{\text{max}} \sim L_1\gtrsim L_2, L_3$, then 
$$
C(N,L)\sim
N^{-1}_{\text{max}}  L^\frac12_{\text{min}} 
\min \left( N_{\text{min}}  N^{2}_{\text{max}},  \frac{ N_{\text{max}} } {N_{\text{min}} }  L_{\text{med}} \right)^\frac12.
$$
Similar estimates hold for any permutations of $(1,2,3)$.
\item  \label{cnl3} In all other cases, we have
          $$
C(N,L)\sim N^{-1}_{\text{max}} L^\frac12_{\text{min}}
\min \left( N_{\text{min}}  N^{2}_{\text{max}},  L_{\text{med}}\right)^\frac12.
$$
\end{enumerate}
\end{proposition}

\begin{remark}\label{rmk-prop1}
In view of \eqref{PP} the bilinear estimate \eqref{BiNLest} still holds if we replace the projection 
 $$\dot P_{N_j} \dot Q_{L_j}   \mathbb{1}_{ \left\{N_j, \ L_j>0\right\}} \quad \text{by} \quad P_{N_j}Q_{L_j}  \mathbb{1}_{ \left\{N_j, \ L_j>1 \right\}}$$ with $C(N,L)$ as in Proposition \ref{prop-dydbiest}\eqref{cnl1}--\eqref{cnl3}. Following the proof of 
 \cite[Proposition 6.1 ]{Tao2001} we also see that \eqref{BiNLest} holds if we 
 we replace $\dot Q_{L_j } \mathbb{1}_{ \left\{L_j>0\right\}}$ by $Q_{L_j} \mathbb{1}_{ \left\{L_j\ge 1 \right\}}$. 
 \end{remark}
 
 In view of Remark \ref{rmk-prop1} we have the following:
\begin{corollary}\label{cor-dydbiest}
Let $N_j>1$ and $L_j\ge 1$  ($j=1,2,3$) be dyadic numbers. The estimate
\begin{equation} \label{INL-est-i}
\norm{ P_{N_3} Q_{L_3} 
\left(  \left(  P_{N_1}  Q_{L_1} u_1 \right) \left(   P_{N_2}  Q_{L_2} u_2\right) \right)} \le C(N,L) \prod_{j=1}^2 \|    P_{N_j}  Q_{L_j} u_j \|
\end{equation}
holds with $C(N,L)$ given as in Proposition \ref{prop-dydbiest}\eqref{cnl1}--\eqref{cnl3}.
\end{corollary}

This Corollary is used to prove the following Lemma.
\begin{lemma}[See \cite{Guo2009}]\label{lm-keybiest}
For dyadic numbers $N_j\ge 1$ ($j=1,2,3$) we have the following: 
\begin{enumerate}[(i)]
\item \label{keybiest1}The bilinear estimate
\begin{equation} \label{keybi-est}
 \norm {  
\Lambda^{-1} P_{N_3} \partial_x \left ( u_{N_1} v_{N_2}  \right) }_{X}  \le C(N)
\norm {  u_{N_1} }_{X}  \norm {  v_{N_2} }_{X} ,
\end{equation}
holds with $C(N)$ as follows:
\begin{equation}
\label{CN-biest}
C(N)\sim 
\begin{cases}
& N_1^{-1} N_2^{-\frac12+}  , \quad \text{if} \  \ N_3 \sim N_2 \gg  N_1>1,
\\
& N_2^{-\frac34}  ,  \quad   \quad   \quad \ \ \text{if} \ \ N_1 \sim N_2 \sim N_3 \gg 1,
\\
& N_1^{-\frac32+},   \quad   \quad  \quad   \ \text{if} \ \ N_1 \sim N_2 \gg N_3=1,
\\
& \max\left( N_1^{-\frac32} , N_1^{-2+}  N_3^{\frac12} \right)  , \quad \text{if} \ \  N_1 \sim N_2 \gg N_3>1.
\end{cases}
\end{equation}
\item \label{keybiest2} If $ N_3 \sim N_2 \gg N_1=1 $ then
\begin{equation}
\label{hlh-est}
  \norm {  
\Lambda^{-1} P_{N_3} \partial_x \left ( u_{N_1} v_{N_2}  \right) }_{X}     \lesssim  \norm {  u_{N_1}  }_{L^2_xL_t^\infty} 
\norm {  v_{N_2} }_{X} .
\end{equation}
\item \label{keybiest3} Let $I$ be a bounded interval. If $1\le N_1, N_2 , N_3\lesssim 1$, then
\begin{equation}
\label{lll-est}
\norm {  
 \mathbb{1}_I (t) \Lambda^{-1} P_{N_3} \partial_x  \left ( u_{N_1} v_{N_2}   \right) }_{X} \lesssim  \norm {  u_{N_1} }_{L_t^\infty L_x^2} 
\norm {  v_{N_2}  }_{L_t^\infty L_x^2} .
\end{equation}

\end{enumerate}

\end{lemma}
For completeness, and since the notation and setup of this paper is slightly different from \cite{Guo2009} we include the proof of Lemma \ref{lm-keybiest} in Appendix A.

In the case of high-high:low frequency interaction, i.e., when $ N_1 \sim N_2 \gg N_3=1$ the factor $C(N)\sim N_1^{-3/2+}$ (see third line in \eqref{CN-biest}) in the dyadic bilinear estimate \eqref{keybi-est} is not good enough to obtain \eqref{biest2}. Fortunately, Guo improved  this estimate to $C(N)\sim N_1^{-3/2}$ which is given as follows.
\begin{lemma} [  \cite{Guo2009}:  $L_x^2 L_t^\infty$-estimate] \label{lm-lowfreqbiest} 
Assume $N_1\sim N_2\gg 1$.  
\begin{enumerate}[(i)]
\item \label{lowfreqbiest1}  Let
$u_{N_1}(t)=S(t)f_{N_1} $ and $ v_{N_2}(t)=S(t)g_{N_2} $
be two free solutions of the Airy equation, where $f_{N_1}$, $ g_{N_2}\in L^2$.
Then
\begin{equation}\label{x1est1}
\norm{\chi(t) \int_0^t S(t-t')  P_1 \partial_x\left[ (u_{N_1}  v_{N_2})(t') \right]\, dt'}_{L_x^2 L_t^\infty } 
\lesssim N_1^{-\frac32} \norm{ f_{N_1}}_{L^2_x} \norm{ g_{N_2}}_{L^2_x}.
\end{equation}

\item \label{lowfreqbiest2} For all $u_{N_1} , v_{N_2} \in  X$, we have by the transfer principle (see e.g. \cite [Lemma 3.2]{Guo2009} )
\begin{equation}\label{x1est2}
\norm{\chi(t) \int_0^t S(t-t')  P_1 \partial_x \left[ (u_{N_1}  v_{N_2})(t') \right]\, dt'}_{L_x^2 L_t^\infty } 
\lesssim N_1^{-\frac32} \norm{ u_{N_1}}_{X} \norm{ v_{N_2}}_{X}.
\end{equation}
\end{enumerate}

\end{lemma}

Lemma \ref{lm-keybiest} and Lemma \ref{lm-lowfreqbiest} together are key to obtain the bilinear estimate \eqref{biest2} which is used  to prove the end-point well-posedness result for \eqref{kdv}. We include the proof of the following Lemma (the proof will be reused  later).
\begin{lemma}[See \cite{Guo2009}]\label{lm-bilinearest}
Define the bilinear operator 
$$
\mathcal B(u, v)(t)=\chi(t/4) \int_0^t S(t-t') \partial_x \left( [ \chi u \cdot  \chi v ](t') \right) \, dt'.
$$
Assume $s\in[-3/4, 0]$. Then for all $u, v \in \bar X^{s} $ we have
\begin{equation}\label{bilinearest}
\norm{ \mathcal B (u, v)}_{\bar X^{s}} 
\lesssim \left(\norm{ u}_{\bar X^{s }} \norm{ v}_{\bar X^{-\frac34}} + \norm{ u}_{\bar X^{-\frac34} }  \norm{ v}_{\bar X^s}  \right).
\end{equation}
\end{lemma}
\begin{proof}
By definition
\begin{align*}
\norm{    \mathcal B(u, v) }^2_{ \bar X^{s} } &=  \mathcal I_1 + \mathcal I_2,
\end{align*}
where
\begin{align*}
 \mathcal I_1&=  \norm{P_1   \mathcal B(u, v)}^2_{L_x^2 L^\infty_t},
\\
\mathcal I_2&=\sum_{N_3 >
1}N_3 ^{2s} \norm{P_{N_3}   \mathcal B(u, v)}^2_{X} .
\end{align*}

\subsection*{Estimate for $\mathcal I_1$}
It suffices to show 
\begin{equation}\label{I1}
\mathcal I _1  
 \lesssim \norm{ u }^2_{\bar X^{-\frac34 } } \norm{ v }^2_{\bar X^{-\frac34 } } .
\end{equation}
Decomposing $u$ and $v$ we have
\begin{align*}
\mathcal I_1&\lesssim \left(  \sum_{N_1,N_2 \ge 1
}\norm{P_1  \mathcal B(u_{N_1},v_{N_2})}_{L_x^2 L^\infty_t } \right)^2  .
\end{align*}

By symmetry we may assume $N_1\le N_2$.  If $N_2\lesssim 1$ we use Lemma \ref{lm-embed}\eqref{emb2}, Lemma \ref{lm-linearest}\eqref{linearest2},
Lemma \ref{lm-keybiest}\eqref{keybiest3} and  Lemma \ref{lm-embed}\eqref{emb1} to obtain
\begin{align*}
\mathcal I _1  &\lesssim \left (\sum_{1\le N_1\le N_2 \lesssim 1
} 
 \norm{ \Lambda^{-1} P_1 \partial_x( \chi u_{N_1} \cdot \chi v_{N_2})}_{X }  \right)^2
 \\
&\lesssim \left( \sum_{1\le N_1\le N_2 \lesssim 1
} 
 \norm{ u_{N_1} }_{L_t^\infty L_x^2} \norm{ v_{N_2} }_{L_t^\infty L_x^2} \right)^2
 \\
 &\lesssim \norm{ u }^2_{\bar X^{-\frac34 } } \norm{ v }^2_{\bar X^{-\frac34 } } .
\end{align*}

If $N_1\sim N_2\gg 1$, then by Lemma \ref{lm-lowfreqbiest}\eqref{lowfreqbiest2}, and Cauchy Schwarz in  $N_1\sim N_2$ we have
\begin{align*}
\mathcal I _1  &\lesssim \left(  \sum_{ N_1\sim N_2 \gg 1
}  N_1^{-\frac32}
 \norm{ \chi u_{N_1} }_{ X} \norm{\chi  v_{N_2} }_{X} \right)^2
 \\
 &\lesssim \norm{ u }^2_{\bar X^{-\frac34 } } \norm{ v }^2_{\bar X^{-\frac34 } } .
\end{align*}

\begin{remark}\label{logloss}  In the case $N_1\sim N_2 \gg N_3=1$ if we use
 Lemma \ref{lm-keybiest}\eqref{keybiest1} with $C(N)$ as in the third line of \eqref{CN-biest} instead of Lemma \ref{lm-lowfreqbiest}\eqref{lowfreqbiest2} we would obtain
\begin{align*}
\mathcal I_1&\lesssim \left(  \sum_{ N_1\sim N_2 \gg 1
}  N_1^{-\frac32+}
 \norm{  u_{N_1} }_{ X} \norm{ v_{N_2} }_{X} \right)^2
 \\
 &\lesssim \norm{ u }^2_{\bar X^{-\frac34+ } } \norm{ v }^2_{\bar X^{-\frac34+ } } .
\end{align*}
Thus, using Lemma \ref{lm-keybiest} in the case of high:high-low  frequency interaction case introduces a logarithmic loss in the estimate for \eqref{biest2}.
\end{remark}


\subsection*{Estimate for $ \mathcal I_2$} 
We want to show 
\begin{equation}\label{I2}
\mathcal I _2  
 \lesssim\left( \norm{ u}_{\bar X^{s }} \norm{ v}_{\bar X^{-\frac34}} + \norm{ u}_{\bar X^{-\frac34} }  \norm{ v}_{\bar X^s}  \right).
\end{equation}
Decomposing $u$ and $v$, and using Lemma \ref{lm-linearest}\eqref{linearest2} we obtain
\begin{align*}
\mathcal I_2&=\sum_{N_3 >
1} \left(  \sum_{N_1,N_2 \ge 1
} N_3 ^{s}\norm{P_{N_3}  \mathcal B(u_{N_1},v_{N_2})}_{X} \right)^2 
\\
&=\sum_{N_3 >
1} \left ( \sum_{N_1,N_2 \ge 1
} N_3 ^{s}\norm{ \Lambda^{-1} P_{N_3}  \partial_x( \chi u_{N_1} \cdot \chi v_{N_2})}_{X } \right)^2 
\\
&:=\mathcal I_{3} + \mathcal I_{4},
\end{align*}
where
\begin{align*}
\mathcal I_{3}&=\sum_{ N_3>1} \left ( \sum_{  N_2 \ge  N_1\ge 1} (\cdot)\right)^2, \quad  \mathcal I_{4}=\sum_{ N_3>1} \left ( \sum_{   \  N_1 \ge  N_2  \ge 1}  (\cdot) \right)^2 .
\end{align*}
Here $(\cdot)$ represents the argument in the inner summation. 

By symmetry we may only estimate $\mathcal I_3$. Thus, it suffices to prove
\begin{equation*}
\mathcal I_{3}\lesssim \norm{ u }^2_{\bar X^{-\frac34 } } \norm{ v }^2_{\bar X^{s }} .
\end{equation*}
In view of \eqref{Ncompare} this reduces further to
\begin{equation}\label{I3}
\mathcal I_{3k}\lesssim \norm{ u }^2_{\bar X^{-\frac34 } } \norm{ v }^2_{\bar X^{s }}  \quad  (k=1,\cdots, 5) ,
\end{equation}
where
\begin{align*}
\mathcal I_{31}&= \sum_{ N_3\sim 1} \left ( \sum_{ 1\le N_1\le  N_2\lesssim  1} (\cdot)\right)^2,  
\quad 
\mathcal I_{32}= \sum_{ N_3>1 } \left ( \sum_{ 1= N_1\ll N_2\sim N_3 } (\cdot)\right)^2,  
\\
\mathcal I_{33}&= \sum_{ N_3>1 } \left ( \sum_{ 1< N_1\ll N_2\sim N_3 } (\cdot)\right)^2,  
\ \ \mathcal I_{34}= \sum_{ N_3\gg 1 } \left ( \sum_{  N_1\sim N_2\sim N_3 } (\cdot)\right)^2,  
\quad 
\mathcal I_{35}= \sum_{ N_3>1 } \left ( \sum_{  N_1\sim N_2\gg N_3 } (\cdot)\right)^2.
\end{align*}

We establish \eqref{I3} as follows.

(i). $\mathcal I_{31}$: 
By Lemma \ref{lm-keybiest}\eqref{keybiest3} and Lemma \ref{lm-embed}\eqref{emb1} we have 
\begin{align*}
\mathcal I_{31}&\lesssim  \sum_{N_3 \sim 1
} \left(  \sum_{\ 1\le N_1 \le N_2\lesssim 1 } N_3 ^{s} \norm{ u_{N_1}}_{L_t^\infty L_x^2 } \norm{ v_{N_2}
 }_{L_t^\infty L_x^2 }\right)^2
 \\
 &\lesssim \norm{ u }^2_{\bar X^{-\frac34 } } \norm{ v }^2_{\bar X^{-\frac34 }} .
\end{align*}

(ii). $\mathcal I_{32}$: 
By Lemma \ref{lm-keybiest}\eqref{keybiest2} and \eqref{dyadicsum} we have 
\begin{align*}
\mathcal I_{32}&\lesssim  \sum_{N_3 \gg 1
} \left(  \sum_{1= N_1\ll N_2 \sim N_3
}  N_3 ^{s}  \norm {  u_{N_1}  }_{L^2_xL_t^\infty} 
\norm {  v_{N_2} }_{X} \right)^2
 \\
 &\lesssim  \norm{ u }^2_{\bar X^{-\frac34 } }  \sum_{N_3 \gg 1
}N_3 ^{2s} \left( \sum_{N_2 \sim N_3
}    
\norm {  v_{N_2} }_{X} \right)^2
\\
 &\lesssim \norm{ u }^2_{\bar X^{-\frac34 } } \norm{ v }^2_{\bar X^{s }} .
\end{align*}

(iii). $\mathcal I_{33}$: 
By Lemma \ref{lm-keybiest}\eqref{keybiest1} with $C(N)$ as in the first line of \eqref{CN-biest} and \eqref{dyadicsum} we have 
\begin{align*}
\mathcal I_{33}&\lesssim  \sum_{N_3 \gg 1
} \left(  \sum_{1< N_1\ll N_2 \sim N_3
}   N_3 ^{s} N_1 ^{-1}N_2^{-\frac12+}   \norm{ u_{N_1} }_{X } \norm{ v_{N_2} }_{X }  \right)^2
 \\
 &\lesssim  \norm{ u }^2_{\bar X^{-\frac34 } }  \sum_{N_3 \gg 1
}\left(  \sum_{N_2 \sim N_3
}    N_3 ^{s} N_2^{-\frac12+} 
\norm {  v_{N_2} }_{X} \right)^2
\\
 &\lesssim \norm{ u }^2_{\bar X^{-\frac34 } } \norm{ v }^2_{\bar X^{s }} ,
\end{align*}
where to obtain the second inequality we used Cauchy-Schwarz in $N_1$.

(iv). $\mathcal I_{34}$: 
By Lemma \ref{lm-keybiest}\eqref{keybiest1} with $C(N)$ as in the second line of \eqref{CN-biest} and \eqref{dyadicsum} we have 
\begin{align*}
\mathcal I_{34}&\lesssim  \sum_{N_3 \gg 1
} \left(  \sum_{ N_1\sim N_2 \sim N_3 \gg 1
}  N_3 ^{s}  N_1 ^{-\frac34} \norm{ u_{N_1} }_{X } \norm{ v_{N_2} }_{X }  \right)^2
 \\
 &\lesssim  \norm{ u }^2_{\bar X^{-\frac34 } }  \sum_{N_3 \gg 1
} \left( \sum_{N_2 \sim N_3
}    N_3 ^{s}
\norm {  v_{N_2} }_{X} \right)^2
\\
 &\lesssim \norm{ u }^2_{\bar X^{-\frac34 } } \norm{ v }^2_{\bar X^{s }} ,
\end{align*}
where to obtain the second inequality we used Cauchy-Schwarz in $N_1\sim N_2$.

(v).  $\mathcal I_{35}$: 
By Lemma \ref{lm-keybiest}\eqref{keybiest1} with $C(N)$ as in the fourth line of \eqref{CN-biest} and \eqref{dyadicsum} we have 
\begin{align*}
\mathcal I_{35}&\lesssim  \sum_{N_3 > 1
} \left(  \sum_{ N_1\sim N_2 \gg N_3 
}  N_3 ^{s} \max\left( N_1^{-\frac32} , N_1^{-2+}  N_3^{\frac12} \right) \norm{ u_{N_1} }_{X } \norm{ v_{N_2} }_{X }  \right)^2
 \\
 &\lesssim  \norm{ u }^2_{\bar X^{-\frac34 } }   \norm{ v }^2_{\bar X^{s }} \sum_{N_3 > 1
}N_3 ^{2s}  N_3 ^{-2s -\frac34+} 
\\
 &\lesssim \norm{ u }^2_{\bar X^{-\frac34 } } \norm{ v }^2_{\bar X^{s }} ,
\end{align*}
where to obtain the second inequality we used
the fact that $$\\max\left( N_1^{-\frac32} , N_1^{-2+}  N_3^{\frac12} \right) \lesssim  N_3^{-s-\frac34+} N_1^{-\frac34}  N_2^{s} $$ 
and Cauchy-Schwarz in $N_1\sim N_2$.

 \end{proof}



\section{Estimates in  Gevrey type spaces and well-posedness of \eqref{kdv}}


The Bourgain-Gevrey type space, denoted $\bar X^{\sigma, s}$,  is defined with respect to the norm
\begin{align*}
\| u \|_{\bar X^{\sigma, s}} &= \left\| e^{\sigma | D_x |}u \right\|_{\bar X^{s}} .
\end{align*}
When $\sigma=0$ the spaces $\bar X^{\sigma, s}$ coincides with $\bar  X^{s}$. 
The restrictions of $\bar X^{\sigma, s}$ to a time slab $I\times \mathbb R$ is defined in a similar way as before.
 
 \subsection{Linear estimates in Gevery space}
 By substitution $u\rightarrow e^{\sigma | D _x|}u$ and $f\rightarrow e^{\sigma | D _x|}f$ in Lemma \ref{lm-embed}\eqref{emb3}, \eqref{emb4} and Lemma \ref{lm-linearest}\eqref{linearest1}, respectively, we easily get the following.
\begin{lemma}\label{lm-embedgvry}
Let $\sigma \ge 0$, $s \in \mathbb{R}$. Then
\begin{enumerate} [(i)]
\item \label{embedgvry1}   we have $\bar X^{\sigma,s} \subset C(\mathbb R,G^{\sigma,s})$ and
\begin{equation*}
\sup_{t \in \mathbb{R}} \| u(t) \|_{G^{\sigma, s}} \leq C \| u \|_{\bar X^{\sigma, s}}.
\end{equation*}
for some absolute constant $C > 0$.
\item \label{embgvry2} 
for all $s_1\le s_2$ we have $\bar X^{\sigma, s_2} \subset \bar X^{\sigma, s_1}$.

\item \label{embgvry3} for all $f\in G^{\sigma, s}$ there exists a constant $C>0$ such that
\begin{equation*}
\norm{\chi(t) S(t) f}_{\bar X^{\sigma, s}} \le C \norm{ f}_{G^{\sigma, s}}.
\end{equation*}
\end{enumerate}
\end{lemma}

 \subsection{Bilinear estimates in Gevrey space}

From Lemma \ref{lm-bilinearest} and a simple triangle inequality we obtain the following.
\begin{corollary}\label{cor-bilinearest}
Let $\mathcal B (u,v)$ be the bilinear form  in Lemma \ref{lm-bilinearest}. Then for all $u, v \in \bar X^{\sigma, s} $, where  $\sigma\ge 0$ and $s\in[-3/4, 0]$, we have
\begin{equation}\label{bilinearest-xsigma}
\norm{ \mathcal B (u, v)}_{\bar X^{\sigma, s}} 
\lesssim \left(\norm{ u}_{\bar X^{\sigma, s }} \norm{ v}_{\bar X^{\sigma,-\frac34}} + \norm{ u}_{\bar X^{\sigma, -\frac34} }  \norm{ v}_{\bar X^{\sigma, s}}  \right)
\end{equation}
\end{corollary}

\begin{proof}
By definition of the $\bar X^{\sigma, s}$-norm we have
\begin{align*}
\norm{    \mathcal B(u, v) }^2_{ \bar X^{\sigma, s} } &=
 \norm{   e^{\sigma |D_x|} \mathcal B(u, v) }^2_{ \bar X^{s} }\lesssim  \mathcal L_1 + \mathcal L_2,
\end{align*}
where
\begin{align*}
\mathcal L_1&= e^{\sigma} \norm{  P_1   \mathcal B(u, v)}^2_{L_x^2 L^\infty_t},
\\
\mathcal L_2&=\sum_{N_3 >
1}N_3 ^{2s} e^{2\sigma N_3}\norm{ P_{N_3}    \mathcal B(u, v)}^2_{X} .
\end{align*}
But using the estimate for $\mathcal I_1$ in \eqref{I1} we obtain
\begin{align*}
\mathcal L_1&\lesssim e^{2\sigma}  \norm{ u}^2_{\bar X^{-\frac34}} \norm{ v}^2_{\bar X^{-\frac34}} 
\lesssim   \norm{ u}^2_{\bar X^{\sigma, -\frac34}} \norm{ v}^2_{\bar X^{\sigma, -\frac34}} .
\end{align*}

Decomposing $u$ and $v$ we have
\begin{align*}
\mathcal L_2&\lesssim \sum_{N_3 >
1} \left(  \sum_{N_1,N_2 \ge 1
} N_3 ^{s} e^{\sigma N_3} \norm{P_{N_3}  \mathcal B(u_{N_1},v_{N_2})}_{X} \right)^2 .
\end{align*}
Let $U=e^{\sigma |D_x|}  u$ and  $V=e^{\sigma |D_x|}  v$.
Since $N_3\lesssim N_1+ N_2$, by the triangle inequality,  it follows that
$$
e^{\sigma N_3} \lesssim e^{\sigma N_1} e^{\sigma N_2} 
$$
 which can be combined with the estimate for $\mathcal J_2$ in \eqref{I2}  above to obtain
\begin{align*}
\mathcal L_2&\lesssim \sum_{N_3 >
1}  \left(  \sum_{N_1,N_2 \ge 1
} N_3 ^{s} \norm{ P_{N_3}  \mathcal B\left( U_{N_1} , V_{N_2}  \right)}_{X} \right)^2
\\
&\lesssim \norm{ U}^2_{\bar X^{ s }} \norm{ V}^2_{\bar X^{-\frac34}} + \norm{ U}^2_{\bar X^{-\frac34} }  \norm{ V}^2_{\bar X^{ s}}  
\\
&= \norm{ u}^2_{\bar X^{\sigma, s }} \norm{ v}^2_{\bar X^{\sigma,-\frac34}} + \norm{ u}^2_{\bar X^{\sigma, -\frac34} }  \norm{ v}^2_{\bar X^{\sigma, s}}  .
\end{align*}

\end{proof}

\subsection{Local well-posedness in Gevery class}
Define the map
$$
\Phi(u)(t)=\chi(t/4) W(t) f + \mathcal B(u, u)(t).
$$
Let $s\in [-3/4, 0]$ and $J=[-1,1]$. By Lemma \ref{lm-embedgvry}\eqref{embgvry2}, \eqref{embgvry3}  and Corollary \ref{cor-bilinearest} we have
\begin{align*}
\norm{ \Phi(u) }_{\bar X^{\sigma, s}_J} \le C\left(  \norm{ f}_{G^{\sigma, s}}+  \norm{ u}^2_{ \bar X^{\sigma, s}_J}\right).
\end{align*}
A similar estimate can be derived for the difference $\Phi(u) -\Phi(v) $, where $v$ is also a solution. Then by a standard fixed point argument \eqref{kdv} admits a unique solution 
 $$ u \in \bar X^{\sigma, s}_J \subset  C\left( J; G^{\sigma, s}(\R)\right)$$
provided the data norm $ \norm{ f}_{G^{\sigma, s}}$ is sufficiently small. Moreover,  the data to solution map $f\ \mapsto u $ is Lipschitz continuous from $\left\{f \in G^{\sigma, s}:  \norm{ f}_{G^{\sigma, s}}\le \epsilon \right\}$ to $ C\left( I; G^{\sigma, s}(\R)\right)$. Moreover, the solution $u$ satisfies the bound 
$\| u\|_{\bar X^{\sigma,  s}_J }\lesssim \| f\|_{ G^{\sigma, s}}.
$

Finally, a local solution for \eqref{kdv} with arbitrarily large $\norm{ f}_{G^{\sigma, s}}$ can be constructed using the scaling symmetry of KdV. Indeed, observe that if
$u$ solves \eqref{kdv} so does 
\begin{equation}\label{scaling}
u_\lambda(t,x) = \lambda^2 u\left(\lambda^3 t, \lambda x\right) 
\end{equation}
with initial data 
$f_\lambda(x) = \lambda^2 f\left( \lambda x\right)$ for some $\lambda>0$. Now given $f$ with arbitrarily large $\norm{ f}_{G^{\sigma, s}}$ one can choose $\lambda$ to be arbitrarily small ($0<\lambda\ll 1$) that
\begin{equation}\label{scalednorm}
\begin{split}
 \norm{f_\lambda}_{G^{\sigma, s}} &= \lambda^\frac32 \left(\int_\R e^{\lambda \sigma|\xi| } \angles{\lambda \xi}^{2s} |\widehat f(\xi)|^2 \, d\xi \right)^\frac12
 \\
 &\lesssim
  \lambda^{\frac32+s} \norm{ f}_{G^{\sigma, s}} \ll 1.
\end{split} 
\end{equation}

By the above argument on local existence theory there exists a solution 
$ u_\lambda \in  C\left(J; G^{\sigma, s} (\R)\right) $ to \eqref{kdv} with initial data $u_\lambda(0)=f_\lambda$. 
By the scaling \eqref{scaling}
 $ u  $ solves \eqref{kdv} on $I\times \R$, where $I=\lambda^3 J=[-\lambda^3, \lambda^3]$. In view of \eqref{scalednorm} the time of existence is given by
$$
t_0:=\lambda^3= c \left[ \norm{ f}_{G^{\sigma, s} }\right]^{-\frac{6}{3+2s}},
$$ 
for some $0<c\ll 1$. 

In conclusion, we have the following local well-posedness result in Gevrey class.
\begin{theorem}[Local well-posedness]\label{thm-lwp} 
Let $\sigma > 0$ and $s\in [-3/4, 0]$. Then for any $f  \in G^{\sigma,s}$ there exists a time 
$$t_0= C_0( \norm{ f}_{G^{\sigma, s} })>0$$
and a unique solution $u$ of \eqref{kdv} on the time interval $I=[-t_0, t_0]$ such that
\[
u \in C(I;  G^{\sigma,s}).
\]
Moreover, the solution depends continuously on the data $f$, and satisfies the bound 
\begin{equation} \label{solnbound1}
\| u\|_{\bar X^{ \sigma, s}_I}\le C \| f\|_{G^{\sigma, s}},
\end{equation}  
where $C$ depends only on $s$. In particular, Lemma \ref{lm-embedgvry}\eqref{embedgvry1} and \eqref{solnbound1} gives the bound
\begin{equation}\label{solnbound2}
 \sup_{t\in I} \| u(t) \|_{G^{\sigma , s}} \le C\| f\|_{G^{\sigma, s}} .
\end{equation} 
\end{theorem}

\begin{remark}\label{rmk-lwp}
Theorem \ref{thm-lwp} shows that if the initial data $f$ is analytic on the strip $S_\sigma$ so is the solution $u(t)$
 on the same strip as long as $t\in I$. Note also that in view of the embedding \eqref{Gembedding} we can allow all $s\in \R $ in Theorem \ref{thm-lwp} but then 
 the solution will be analytic only on a slightly smaller strip $S_{\sigma-}$.
 \end{remark}


\section{Almost conservation law and lower bound for $\sigma$}
\subsection{Almost conservation law}
For a given $u(0)=f\in G^{\sigma} $ we have by the above local existence theory a solution $u(t)\in  G^{\sigma}$ for $0\le t\le t_0$, where (setting $s=0$ in Theorem \ref{thm-lwp})
\begin{equation}\label{t0}
t_0=C_0(  \norm{ f}_{G^{\sigma} })>0
\end{equation}
The solution $u$
satisfies the bound
\begin{equation}\label{C}
\sup_{t\in [0,  t_0]} \| u(t) \|_{G^{\sigma}}\le C\| u(0) \|_{G^{\sigma}},
\end{equation}
where the constant $C$ in \eqref{C} comes from \eqref{solnbound2} and is independent of $t_0$ and $\sigma$. 
The question is then whether we can improve on estimate \eqref{C}. In what follows
we will use equation \eqref{kdv} and the local existence theory mentioned above to obtain the approximate conservation law
\begin{equation*}
 \sup_{t\in [0,  t_0]} \| u(t) \|^2_{G^{\sigma}} =\| u(0) \|^2_{G^{\sigma}} +\mathcal  E_\sigma(0) ,
\end{equation*} 
 where $\mathcal E_\sigma(0)$ satisfies the bound 
 $\mathcal E_\sigma(0)\le C \sigma^{\frac34}  \| u(0) \|^3_{G^{\sigma}}$. The quantity $\mathcal E_\sigma(0)$ can be considered an error term since in the limit as $\sigma\rightarrow 0 $, we have 
 $\mathcal E_\sigma(0) \rightarrow 0$, and hence recovering the well-known conservation of $L^2$-norm of solution: 
 $$ \| u(t) \|_{L^2_x}=  \| u(0) \|_{L^2_x} \quad \text{for all} \ \  t\in  [0,  t_0].$$

\begin{theorem}\label{thm-approx}
Let $I=[0,t_0]$, where $t_0$ is as in \eqref{t0}. Then there exists $C > 0$ such that for any $\sigma > 0$ and any solution
 $u \in  \bar X^{\sigma, 0}_I$ to the Cauchy problem \eqref{kdv} on the time interval  $I$, we have the estimate
\begin{equation}\label{approx1}
\sup_{t\in [0,t_0]} \| u(t) \|^2_{G^{\sigma}} \leq \| u(0) \|^2_{G^{\sigma}} + C \sigma^{\frac 34} \| u \|^3_{\bar X^{\sigma, 0}_{I}}.
\end{equation} 
Moreover, we have
\begin{equation}\label{approx2}
 \sup_{t\in [0,t_0]} \| u(t) \|^2_{G^{\sigma}} \leq \| u(0) \|^2_{G^{\sigma}} + C \sigma^{\frac34} \| u(0) \|^3_{G^{\sigma}},
\end{equation} 

\end{theorem}
\begin{proof}
The estimate \eqref{approx2} follows from \eqref{approx1} and \eqref{solnbound1} with $s=0$. Thus, it remains to prove \eqref{approx1}. To this end
set 
$$w=e^{\sigma |D_x|}u,
$$ 
where $u$ is the solution to \eqref{kdv}. Since $u$ is real-valued so is $w$. Now set
\begin{equation}\label{fw}
f(w)=\frac12 \partial_x \left[  w\cdot w  -e^{\sigma |D_x| }\left( e^{-\sigma |D_x| }w \cdot   e^{-\sigma |D_x| }w  \right) \right].
\end{equation}
Then we use \eqref{kdv} to obtain
\begin{equation}\label{kdvm}
w_t + w_{xxx} + w w_x  = f(w).
\end{equation}

Multiplying \eqref{kdvm} by $w$ and integrating in space we obtain
$$
\frac12 \frac{d}{dt} \int_\R w^2 dt  + \int_\R  \partial_x\left( w w_{xx} - \frac12 w_x^2 +\frac13 w^3\right) dx= \int_\R  wf(w)  dx.
$$
We may assume $w, w_x$ and $w_{xx}$ decays to zero as $|x|\rightarrow \infty$. This in turn implies
$$
\frac{d}{dt}\int_\R w^2 dx= 2\int_\R wf(w) dx.
$$
Now integrating in time over the interval $I=[0, t_0]$,  where $t_0\le 1$, we obtain
\begin{align*}
\int_\R w^2(t_0, x) dx= \int_\R  w^2(0,x ) dx+2\int_{\R^{1+1} } \mathbbm{1}_I (t) wf(w) \, dt  dx.
\end{align*}
We conclude that
\begin{equation}
\label{almostcons-est1}
\| u(t_0)\|_{G^\sigma}^2\le\| u(0)\|_{G^\sigma}^2 + \mathcal R,
\end{equation}
where 
\begin{equation*}
\mathcal R=2\left|\int_{\R^{1+1} } \mathbbm{1}_I  (t) w f(w) \, dt dx \right|.
\end{equation*}
Now combining \eqref{almostcons-est1} with the estimate for $\mathcal R$ in \eqref{Rest} below and using  
$$\norm{ w}_{\bar X^0_I}= \norm{ u}_{\bar X^{ \sigma,0}_I}$$ 
we obtain \eqref{approx1}.
\end{proof}

The proof for the following Lemma is given in the next section.
\begin{lemma}\label{lm-Rest} For all $\sigma\ge 0$ and $w\in  {\bar X^0_I}$ we have
\begin{equation}\label{Rest}
\mathcal R \le C   \sigma^\frac34 \norm{ w}_{\bar X^0_I }^3.
\end{equation}
\end{lemma}

\subsection{Lower bound for $\sigma$}

Let $f=u(0) \in G^{\sigma_0 }$ for some $\sigma_0 > 0$.
To construct a solution on $[0, T]$ for arbitrarily large $T$ we can apply
 the approximate conservation law \eqref{approx2} so as to repeat the local result 
on successive short time intervals to reach $T$, by adjusting the strip width parameter $\sigma$ according to the size of $T$.
By employing this strategy one can show that the solution $u$ to \eqref{kdv} satisfies
\begin{equation}
\label{uT}
  u(t) \in G^{\sigma(t) } \quad \text{for all }  \  t\in [0,  T] ,
 \end{equation}
with
\begin{equation}
\label{siglb}
\sigma(t) \ge  c T^{-\frac43}  ,
\end{equation}
where $c > 0$ is a constant depending on $\|f\|_{G^{\sigma_0}}$ and $\sigma_0$.

For completeness we include the proof \eqref{uT}--\eqref{siglb} here which is similar to that of \cite{SD2015}. 
 Define
\[
  \Gamma_\sigma(t) = \| u(t) \|_{G^{\sigma}},
\]
where $\sigma\in (0, \sigma_0]$ is a parameter to be chosen later.
By the local existence theory (see Theorem \ref{thm-lwp}) there is 
a solution $u$ to \eqref{kdv} satisfying
\[
  u(t) \in G^{\sigma_0} \quad \text{for all }  \  t\in [0,  t_0] 
\]
where $$ t_0 = C_0\left( \Gamma_{\sigma_0}(0)  \right).$$

Now fix $T$ arbitrarily large. We shall apply the above local result and \eqref{approx2} repeatedly, with a uniform time step $t_0$, and prove 
\begin{equation}\label{keybound}
\sup_{t\in [0, T]}  \Gamma_{\sigma}(t) \le 2\Gamma_{\sigma_0} (0)
\end{equation}
for $\sigma$ satisfying \eqref{siglb}. 
Hence we have $\Gamma_\sigma(t) < \infty$ for all $t \in [0,T]$, which in turn implies $u(t) \in G^{\sigma(t) }$, and this completes the proof of \eqref{uT}--\eqref{siglb}.

It remains to prove \eqref{keybound}. Choose $n \in \mathbb N$ so that $T \in [nt_0,(n+1)t_0)$. Using induction we can show for 
any $k \in \{1,\dots,n+1\}$ that
\begin{align}
  \label{induction1}
  \sup_{t \in [0, kt_0]} \Gamma^2_\sigma(t) &\le \Gamma^2_\sigma(0) + k 2^\frac32 C\sigma^\frac34   \Gamma^3_{\sigma_0}(0),
  \\
  \label{induction2}
  \sup_{t \in [0,kt_0]} \Gamma^2_\sigma(t) &\le 2\Gamma^2_{\sigma_0}(0),
\end{align}
provided $\sigma$ satisfies 
\begin{equation}\label{sigma}
  \frac{2T}{t_0}  2^\frac32 C\sigma^\frac34 \Gamma_{\sigma_0}(0) \le 1.
\end{equation}

Indeed, for $k=1$, we have from \eqref{approx2} that
\begin{align*}
  \sup_{t \in [0, t_0]} \Gamma^2_\sigma(t) &\le \Gamma^2_\sigma(0) +  C\sigma^\frac34 \Gamma^3_{\sigma}(0)
  \\
  &\le \Gamma^2_\sigma(0) +  
  C\sigma^\frac34 \Gamma^3_{\sigma_0}(0),
\end{align*}
where we used $\Gamma_\sigma(0) \le \Gamma_{\sigma_0}(0) $. This in turn implies \eqref{induction2} provided 
$$C\sigma^\frac34 \Gamma_{\sigma_0}(0) \le 1$$
 which holds by \eqref{sigma} since 
$T>t_0$.

Now 
assume \eqref{induction1} and \eqref{induction2} hold for some $k \in \{1,\dots,n\}$. Then applying \eqref{approx2}, \eqref{induction2} and \eqref{induction1}, respectively, we obtain
\begin{align*}
  \sup_{t \in [kt_0, (k+1)t_0]} \Gamma^2_\sigma(t) &\le \Gamma^2_\sigma(kt_0) +  C\sigma^\frac34  \Gamma^3_{\sigma}(kt_0) 
  \\
   &\le \Gamma^2_\sigma(kt_0) +  2^\frac32 C\sigma^\frac34 \Gamma^3_{\sigma_0}(0)
   \\
      &\le   \Gamma^2_\sigma(0) +  (k+1) 2^\frac32 C\sigma^\frac34  \Gamma^3_{\sigma_0}(0).
\end{align*}
Combining this with the induction hypothesis
 \eqref{induction1} (for $k$) we obtain
 \begin{align*}
  \sup_{t \in [0, (k+1)t_0]} \Gamma^2_\sigma(t) 
      &\le   \Gamma^2_\sigma(0) + (k+1) 2^\frac32 C\sigma^\frac34  \Gamma^3_{\sigma_0}(0).
\end{align*}
 which proves \eqref{induction1} for $k+1$. This also implies \eqref{induction2} for $k+1$ provided
 \begin{align*}
   ( k+1)   2^\frac32 C\sigma^\frac34 \Gamma_{\sigma_0}(0) \le 1.
\end{align*}
 But the latter follows from \eqref{sigma} since 
 $$
 k+1\le n+1\le \frac T{t_0}+ 1 \le \frac {2T}{t_0}.
 $$

Finally,  the condition \eqref{sigma} is satisfied for $\sigma$ such that
    $$
    \frac{2 T}{t_0} 2^\frac32 C\sigma^\frac34 \Gamma_{\sigma_0}(0)=1.
    $$
This implies
$$
\sigma=c_0 T^{-\frac43}, \quad \text{where} \quad c_0=  \left[\frac{ C_0\left( \Gamma_{\sigma_0}(0) \right)}{ 2^\frac52C  \Gamma_{\sigma_0}(0) } \right]^\frac43.
$$
Thus, by choosing $c$ such that $c\le c_0$ we obtain \eqref{siglb}.


\section{Proof of Lemma \ref{lm-Rest} }
We recall that
\begin{equation*}
\mathcal R=2\left|\int_{\R^{1+1} } \mathbbm{1}_I  (t)\cdot  w f(w) \, dt dx \right|,
\end{equation*}
where 
$$
f(w)=\frac12 \partial_x \left[  w\cdot w  -e^{\sigma |D_x| }\left( e^{-\sigma |D_x| }w \cdot   e^{-\sigma |D_x| }w  \right) \right].
$$
By Plancherel and \eqref{betasum} we have 
\begin{align*}
\mathcal R&=2\left|\int_{\R^{1+1} } \mathbbm{1}_I  (t)\cdot  \widehat w(\xi) \widehat{ f(w) }(-\xi) \, dt d\xi \right|,
\\
&=2\left| \sum_{N_3\ge 1}\int_{\R^{1+1} } \mathbbm{1}_I  (t)\cdot  \beta_{N_3}(\xi) \widehat w(\xi)  \beta_{N_3}(\xi)\widehat{ f(w) }(-\xi) \, dt d\xi \right|
\\
&\lesssim  \sum_{N_3\ge 1}\left|\int_{\R^{1+1} } \mathbbm{1}_I  (t)
\left[  P_{N_3} w \cdot P_{N_3} f(w)\right]  \, dt dx \right|
\\
&=\mathcal R_1 +\mathcal R_2,
\end{align*}
where
\begin{align*}
\mathcal R_1&=  \norm{ \mathbbm{1}_{I} (t)( P_1 w \cdot P_1 f(w) )}_{L^1_{t,x}},
\\
\mathcal R_2 &= \sum_{N_3> 1} \norm{ \mathbbm{1}_{I} (t)( P_{N_3}w \cdot P_{N_3} f(w) )}_{L^1_{t,x}}.
\end{align*}

Moreover, since $$w=\sum_{N\ge 1} w_{N}$$
 one can also write
\begin{align*}
2 P_{N_3}&f(w)
\\
 &=\sum_{N_1, N_2\ge 1}P_{N_3} \partial_x \left[  w_{N_1}  \cdot w_{N_2} -e^{\sigma |D_x| }\left(  e^{-\sigma |D_x| }  w_{N_1} \cdot  e^{-\sigma |D_x| }  w_{N_2} \right)  \right].
\end{align*}
Now taking the Fourier Transform of $P_{N_3}f(w)$ we get
\begin{align*}
&2|\mathcal F_x \left[P_{N_3}f(w) \right](\xi)| 
\\
 & \qquad\le   \beta_{N_3}(\xi) 
\sum_{N_1, N_2\ge 1} \int_{\R^2}  |\xi|
 \left|    1  -e^{-\sigma\left ( | \xi_1 | +| \xi_2 | -| \xi_1+\xi_2 | \right)} \right| \prod_{j=1}^2    \beta_{N_j}(\xi_j)| \widehat{w }( \xi_j)| 
 \, d\mu(\xi),
\end{align*}
 where $d\mu$ is the surface measure
$$
d\mu (\xi)=\delta (
\xi-\xi_1-\xi_2 ) d\xi_1 d\xi_2.
$$
Note that for $r \ge 0$ and $0\le \theta \le 1$ we have the simple  inequality 
$$1-e^{-r}\le r^\theta .$$
Setting $r=\sigma ( | \xi_1 | +| \xi_2 | -| \xi_1+\xi_2 | )\ge 0$ and  $\theta=\frac34$ we obtain
\begin{align*}
  1 - e^{-\sigma ( | \xi_1 | +| \xi_2 | -| \xi_1+\xi_2 | )}& \le
   \sigma^\frac34( | \xi_1 | +| \xi_2 | -| \xi_1+\xi_2 | )^\frac34
   \\
   &= \sigma^\frac34 \left (\frac{(| \xi_1 | +| \xi_2 |) ^2-| \xi_1+\xi_2 |^2}{| \xi_1 | +| \xi_2 | +| \xi_1+\xi_2 | }\right)^\frac34
   \\
   &= \sigma^\frac34 \left (\frac{2(| \xi_1 | | \xi_2 |- \xi_1\xi_2 )}{| \xi_1 | +| \xi_2 | +| \xi_1+\xi_2 | }  \right)^\frac34
   \\
   &\le \sigma^\frac34  \left (2 \min(| \xi_1 |, | \xi_2 |) \right)^\frac34
 \\
   & \sim \sigma^\frac34   \min(N_1,N_2)^\frac34.
\end{align*}
Thus, we have
\begin{align*}
&\left|\mathcal F_x \left[ P_{N_3}f(v) \right](\xi)\right| 
\\
 &\qquad\lesssim \sigma^\frac34  \sum_{N_1, N_2\ge 1}  \min(N_1,N_2)^\frac34 |\xi|  \beta_{N_3}(\xi)    \int_{\R^2}
 \prod_{j=1}^2   \beta_{N_j}(\xi_j)|\widehat{w }( \xi_j)| 
 \, d\sigma( \xi)
 \\
 & \qquad = \sigma^\frac34  \sum_{N_1, N_2\ge 1}  \min(N_1,N_2)^\frac34
\left| \mathcal F_x \left[ P_{N_3}\partial_x \left( w_{N_1}   w_{N_2}    \right) \right](\xi) \right| ,
\end{align*}
and hence by Plancherel
\begin{equation}\label{fwest}
 \| P_{N_3}f(w)\|_{L_x^2}
\lesssim \sigma^\frac34 \sum_{N_1, N_2\ge 1}  \min(N_1,N_2)^\frac34  \|  P_{N_3}\partial_x \left( w_{N_1}   w_{N_2}  \right) \|_{L_x^2}.
\end{equation}

Now we give the estimate for $\mathcal R_1$ and $\mathcal R_2$.
\subsection{Estimate for $\mathcal R_1$}
Recall that  $I=[0, t_0]$, where $ t_0\le 1$. By H\"{o}lder inequality and Lemma \ref{lm-embed}\eqref{emb1} we have
\begin{align*}
\mathcal R_1  &\lesssim \norm{ \mathbbm{1}_{I} (t)( P_1 w \cdot P_1 f(w) )}_{{L^1_{t,x}}}
\\
&\lesssim t_0^\frac12  \norm{  P_1 w}_{ L_t^\infty L_x^2 }\norm{ \mathbbm{1}_{I} (t)P_1 f(w) )}
\\
&\lesssim   \norm{   w}_{ \bar X^0 }\norm{ \mathbbm{1}_{I} (t)P_1 f(w) )}.
\end{align*}

Now we claim that
\begin{equation}\label{R11}
\mathcal R_{11}:=\norm{ \mathbbm{1}_{I} (t)P_1 f(w) )} \lesssim \sigma^\frac34  \norm{ w }^2_{\bar X^{0} }.
\end{equation}
This in turn implies the desired estimate for $\mathcal R_1 $, i.e.,
$$
\mathcal R_1 \lesssim \sigma^\frac34  \norm{ w }^3_{\bar X^{0} } .
$$

Next we prove claim \eqref{R11}.
By \eqref{fwest} we have
\begin{align*}
\mathcal R_{11}&\lesssim \sigma^\frac34   \sum_{N_1, N_2\ge 1} \min(N_1,N_2)^\frac34  \norm{ \mathbbm{1}_{I} (t) P_1 \partial_x  \left( w_{N_1}  \cdot w_{N_2}  \right)}.
\end{align*}
By symmetry we may assume $N_1\le N_2$.
\subsubsection{Case: $1\le N_1 \le N_2\lesssim 1$ }
By 
Sobolev, H\"{o}lder inequality and Lemma \ref{lm-embed}\eqref{emb2}
\begin{align*}
\mathcal R_{11}
&\lesssim t_0^\frac12 \sigma^\frac34   \sum_{1\le N_1\le  N_2\lesssim 1} N_1^
\frac34  \norm{P_1 \left( w_{N_1}  \cdot w_{N_2}  \right)}_{ L_t^\infty L_x^1}
\\
&\lesssim   t_0^\frac12 \sigma^\frac34   \sum_{1\le N_1\le  N_2\lesssim 1}   \norm{ w_{N_1} }_{ L_t^\infty L_x^2}\norm{  w_{N_2} }_{ L_t^\infty L_x^2}
\\
&\lesssim   \sigma^\frac34   \norm{   w}^2_{ \bar X^0 }.
\end{align*}

\subsubsection{Case: $ N_1 \sim N_2\gg 1$ }
Decomposing in modulation and in the output frequency we get
\begin{align*}
\mathcal R_{11}
&\lesssim \sigma^\frac34   \sum_{N_1\sim N_2\gg 1}  \sum_{ L_1, L_2, L_3\ge 1}N_1^
\frac34   \norm{ P_1 Q_{L_3} \partial_x ( P_{N_1} Q_{L_1} w  \cdot P_{N_2} Q_{L_2} w)}
\\
&\lesssim \sigma^\frac34   \sum_{N_1\sim N_2\gg 1, }  \sum_{0<M \le 1}   N_1^
\frac34 M \cdot  \mathcal K_M(N_1, N_2)
\end{align*}
where 
$$
\mathcal K_M(N_1, N_2)=
 \sum_{ L_1, L_2, L_3\ge 1}  \norm{ \dot P_{M} Q_{L_3} ( P_{N_1} Q_{L_1} w \cdot P_{N_2} Q_{L_2} w)}.
$$

Next we show that
\begin{equation}\label{AN12}
\mathcal K_M(N_1, N_2)
\lesssim  N_1^
{-1}   \|   w_{N_1}  \|_{X} \|   w_{N_2}  \|_{X}.
\end{equation}
This in turn implies
\begin{align*}
\mathcal R_{11}
&\lesssim \sigma^\frac34   \sum_{N_1\sim N_2\gg 1 }    N_1^{-
\frac14}  \|   w_{N_1} \|_{X}\|   w_{N_2} \|_{X}
\\
&\lesssim   \sigma^\frac34   \norm{   w}^2_{ \bar X^0 },
\end{align*}
where we used Cauchy-Schwarz in $N_1\sim N_2$.

Now we prove \eqref{AN12}.
By Proposition \ref{prop-dydbiest}\eqref{cnl3} (see also Remark \ref{rmk-prop1} and Corollary \ref{cor-dydbiest}) we obtain
\begin{align*}
\mathcal K_M(N_1, N_2)
&\lesssim N_1^
{-1}  \sum_{ L_1, L_2, L_3\ge 1} ( L_{\text{min}} L_{\text{med}})^\frac12 \prod_{j=1}^2 \|   P_{N_j} Q_{L_j} w \|.
\end{align*}

By symmetry we may assume $L_1\le L_2$.
If $L_2\ge L_3$, then by \eqref{dyadicsum}
\begin{align*}
\mathcal K_M(N_1, N_2)
&\lesssim N_1^
{-1}      \sum_{ L_1, L_2\ge 1}   \sum_{  L_3 \le L_2}L_1^\frac12  L_3^\frac12   \prod_{j=1}^2 \| P_{N_j} Q_{L_j}  w\|
\\
&\lesssim N_1^
{-1}  \prod_{j=1}^2   \sum_{ L_j\ge 1}  L_j ^\frac12 \|   P_{N_j} Q_{L_j} w \|
\\
&= N_1^
{-1}  \|   w_{N_1}  \|_{X} \|   w_{N_2}  \|_{X}.
\end{align*}

 If $L_2\le L_3$ then $L_3\sim N_1^2 M$, and hence by \eqref{dyadicsum} (with $\alpha\sim \beta$) we obtain
\begin{align*}
\mathcal K_M(N_1, N_2)
&\lesssim N_1^
{-1}  \prod_{j=1}^2   \sum_{ L_j\ge 1}  L_j ^\frac12 \|   P_{N_j} Q_{L_j} w_j \|  \left(\sum_{  L_3 \sim  N_1^2 M} 1 \right)
\\
&= N_1^
{-1}  \|   w_{N_1}  \|_{X} \|   w_{N_2}  \|_{X}.
\end{align*}


\subsection{Estimate for $\mathcal R_2$}
Decomposing in modulation and using Cauchy-Schwarz we obtain
\begin{align*}
\mathcal R_2
 &\lesssim \sum_{N_3> 1} \sum_{ L_3 \ge 1} \norm{ \mathbbm{1}_{I}\left( P_{N_3}Q_{L_3} w \cdot P_{N_3}Q_{L_3} f(w)\right)}_{L^1_{t,x}}
\\
&\lesssim \sum_{N_3> 1}  \left(\sup_{L_3\ge 1} L_3^{\frac12}\norm{ \mathbbm{1}_{I} P_{N_3}Q_{L_3} w } \right) \left(\sum_{L_3\ge 1} L_3^{-\frac12}\norm{ \mathbbm{1}_{I}P_{N_3}Q_{L_3} f(w)}\right)
\\
&\lesssim \sum_{N_3> 1}  \norm{  w_{N_3} }_{ X} 
 \norm{ \Lambda^{-1} P_{N_3} f(w) }_{X}
 \\
 &\lesssim 
 \left( \sum_{{N_3}>1}  \norm{  w_{N_3} }^2_{X} \right)^\frac12 \cdot \left( \sum_{{N_3}> 1}  
 \norm{ \Lambda^{-1} P_{N_3} f(w) }^2_{X} \right)^\frac12
\end{align*}
By \eqref{fwest} we have
\begin{align*}
 &\norm{ \Lambda^{-1} P_{N_3} f(w) }_{X } 
 \\
 & \qquad  \lesssim  \sigma^\frac34  \sum_{N_1,N_2 \ge 1
}\min(N_1,N_2)^\frac34  \norm{ \Lambda^{-1} P_{N_3}  \partial_x( w_{N_1}  \cdot w_{N_2})}_{X }.
\end{align*}

Then
\begin{align*}
\mathcal R_2& \lesssim \sigma^\frac34  \norm{ w }_{\bar X^{0} } \cdot \left( \sum_{{N_3}> 1}  \left(  \sum_{ N_1,  N_2\ge 1 
}\min(N_1,N_2)^\frac34  \norm{ \Lambda^{-1} P_{N_3}  \partial_x( w_{N_1}  w_{N_2} )}_{X }\right)^2
  \right)^\frac12
  \\
  & \lesssim \sigma^\frac34  \norm{ w }_{\bar X^{0} }  \cdot \left( \mathcal R^\frac12_3+ \mathcal R^\frac12_4 \right),
\end{align*}
where
\begin{align*}
\mathcal R_{3}&= \sum_{{N_3}> 1}  \left(  \sum_{ 1\le N_1\le N_2 
}N_1^\frac34  \norm{ \Lambda^{-1} P_{N_3}  \partial_x( w_{N_1}  w_{N_2} )}_{X }\right)^2,
\\
   \mathcal R_{4}&= \sum_{{N_3}> 1}  \left(  \sum_{ N_1\ge  N_2\ge 1 
}N_2^\frac34 \norm{ \Lambda^{-1} P_{N_3}  \partial_x( w_{N_1}  w_{N_2} )}_{X }\right)^2.
\end{align*}

By symmetry we may only estimate $\mathcal R_3$. Thus, it suffices to prove
\begin{equation*}
\mathcal R_{3}\lesssim \norm{ w}^4_{\bar X^{0 }} .
\end{equation*}
In view of \eqref{Ncompare} this reduces further to
\begin{equation}\label{R3}
\mathcal R_{3k}\lesssim   \norm{ w}^4_{\bar X^{0 }} \quad  (k=1,\cdots, 5) ,
\end{equation}
where
\begin{align*}
\mathcal R_{31}&= \sum_{ N_3\sim 1} \left ( \sum_{ 1\le N_1\le  N_2\lesssim  1} (\cdot)\right)^2,  
\quad 
\mathcal R_{32}= \sum_{ N_3>1 } \left ( \sum_{ 1= N_1\ll N_2\sim N_3 } (\cdot)\right)^2,  
\\
\mathcal R_{33}&= \sum_{ N_3>1 } \left ( \sum_{ 1< N_1\ll N_2\sim N_3 } (\cdot)\right)^2,  
\ \
\mathcal R_{34}= \sum_{ N_3\gg 1 } \left ( \sum_{  N_1\sim N_2\sim N_3 } (\cdot)\right)^2,  
\quad 
\mathcal R_{35}= \sum_{ N_3>1 } \left ( \sum_{  N_1\sim N_2\gg N_3 } (\cdot)\right)^2,  
\end{align*}

(i).  $\mathcal R_{31}$: 
By Lemma \ref{lm-keybiest}\eqref{keybiest3} and Lemma \ref{lm-embed}\eqref{emb1} we have 
\begin{align*}
\mathcal R_{31}&\lesssim  \sum_{N_3 \sim 1
} \left(  \sum_{1\le N_1\le N_2 \lesssim 1
}   \norm{ w_{N_1}}_{L_t^\infty L_x^2 } \norm{ w_{N_2}
 }_{L_t^\infty L_x^2 }\right)^2
 \\
 &\lesssim \norm{ w }^4_{\bar X^{0} }  .
\end{align*}

(ii). $\mathcal R_{32}$: 
By Lemma \ref{lm-keybiest}\eqref{keybiest2} and \eqref{dyadicsum} we have 
\begin{align*}
\mathcal R_{32}&\lesssim  \sum_{N_3 \gg 1
} \left ( \sum_{1= N_1\ll N_2 \sim N_3
}    \norm {  w_{N_1}  }_{L^2_xL_t^\infty} 
\norm {  w_{N_2} }_{X} \right)^2
 \\
 &\lesssim  \norm{ w }^2_{\bar X^{0 } }  \sum_{N_3 \gg 1
} \left(  \sum_{N_2 \sim N_3
}    
\norm {  w_{N_2} }_{X} \right)^2
\\
 &\lesssim \norm{ w }^4_{\bar X^{0 } }  .
\end{align*}

(iii). $\mathcal R_{33}$: 
By Lemma \ref{lm-keybiest}\eqref{keybiest1} with $C(N)$ as in the first line of \eqref{CN-biest} and \eqref{dyadicsum} we have 
\begin{align*}
\mathcal R_{33}&\lesssim  \sum_{N_3 \gg 1
} \left(  \sum_{1< N_1\ll N_2 \sim N_3
}   N_1 ^{-\frac14}N_2^{-\frac12+}   \norm{ w_{N_1} }_{X } \norm{ w_{N_2} }_{X }  \right)^2
 \\
 &\lesssim  \norm{ w }^2_{\bar X^{0} }  \sum_{N_3 \gg 1
} \left(  \sum_{N_2 \sim N_3
}     N_2^{-\frac12+} 
\norm {  w_{N_2} }_{X} \right)^2
\\
 &\lesssim \norm{ w }^4_{\bar X^{0 } }  ,
\end{align*}
where to obtain the second inequality we used Cauchy-Schwarz in $N_1$.

(iv).  $\mathcal R_{34}$: 
By Lemma \ref{lm-keybiest}\eqref{keybiest1} with $C(N)$ as in the second line of \eqref{CN-biest} and \eqref{dyadicsum} we have 
\begin{align*}
\mathcal R_{34}&\lesssim  \sum_{N_3 \gg 1
} \left(  \sum_{ N_1\sim N_2 \sim N_3 \gg 1
}    \norm{ w_{N_1} }_{X } \norm{ w_{N_2} }_{X }  \right)^2
 \\
 &\lesssim  \norm{ w }^2_{\bar X^{0} }  \sum_{N_3 \gg 1
} \left(  \sum_{N_2 \sim N_3
}    
\norm {  w_{N_2} }_{X} \right)^2
\\
 &\lesssim \norm{ w }^4_{\bar X^{0 } },
\end{align*}
where to obtain the second inequality we used Cauchy-Schwarz in $N_1\sim N_2$.

(v). $\mathcal R_{35}$: 
By Lemma \ref{lm-keybiest}\eqref{keybiest1} with $C(N)$ as in the fourth line of \eqref{CN-biest} and \eqref{dyadicsum} we have 
\begin{align*}
\mathcal R_{35}&\lesssim  \sum_{N_3 > 1
} \left(  \sum_{ N_1\sim N_2 \gg N_3 
}  N_1^{\frac34}  \max\left( N_1^{-\frac32} , N_1^{-2+}  N_3^{\frac12} \right) \norm{ w_{N_1} }_{X } \norm{ w_{N_2} }_{X }  \right)^2
 \\
 &\lesssim  \norm{ w }^4_{\bar X^{0 } }   \sum_{N_3 > 1}
N_3 ^{-\frac14}  
\\
 &\lesssim \norm{ w }^4_{\bar X^{0} } ,
\end{align*}
where to obtain the second inequality we used Cauchy-Schwarz in $N_1\sim N_2$.


\appendix

\section{Proof of Lemma \ref{lm-keybiest}}
First we prove \eqref{hlh-est} and \eqref{lll-est}.  
 By definition of $X$, H\"{o}lder inequality, \eqref{stre4} we have
\begin{align*}
 \norm {  
\Lambda^{-1} P_{N_3} \partial_x \left ( u_{N_1} v_{N_2}  \right) }_{X}  
 &\lesssim N_2 \sum_{L \geq 
1}L^{-\frac12} \norm{P_{N_3}Q_L \left( u_{N_1} v_{N_2}   \right)}
\\
 &\lesssim N_2  \norm{  u_{N_1} v_{N_2}  }
\\
& \lesssim N_2  \norm{u_{N_1}  }_{L_x^2 L_t^\infty }
\norm{ v_{N_2}  }_{L_x^\infty  L_t^2 }
\\
& \lesssim  \norm{u_{N_1}  }_{L_x^2 L_t^\infty }
\norm{ v_{N_2}  }_{X},
\end{align*}
where we also used the fact that $P_{N}$ and $Q_{L}$ are bounded in $L^2$. Thus, \eqref{hlh-est} is proved.  Similarly,
by definition of $X$, H\"{o}lder and Bernstein's inequality we obtain 
\begin{align*}
 \norm {  
 \mathbb{1}_I (t) \Lambda^{-1} P_{N_3} \partial_x  \left ( u_{N_1} v_{N_2}   \right) }_{X} 
 &\lesssim  \sum_{L \geq 
1}L^{-\frac12} \norm{  P_{N_3}Q_L \left( 
 \mathbb{1}_I (t)  u_{N_1}  v_{N_2} \right)}
\\
&\lesssim \norm{ 
 \mathbb{1}_I (t)   P_{N_3}\left( u_{N_1}  v_{N_2}  \right)}
\\
&\lesssim \norm{   (u_{N_1}  v_{N_2})  }_{L_t^\infty L_x^1}
\\
& \lesssim   \norm{u_{N_1} }_{L_t^\infty  L_x^2}
\norm{ v_{N_2}  }_{L_t^\infty L_x^2 }
\end{align*}
which is \eqref{lll-est}.

To prove \eqref{keybi-est}--\eqref{CN-biest} we repeatedly use Corollary \ref{cor-dydbiest},  Proposition \ref{prop-dydbiest}, the constraints in \eqref{Ncompare} and \eqref{lmaxmed}. 
To this end we set
\begin{equation*} 
 \mathcal J(N) =\norm {  
\Lambda^{-1} P_{N_3} \partial_x \left ( u_{N_1} v_{N_2}  \right) }_{X}  
\end{equation*}
and 
denote 
$$  u_{N_1, L_1}=P_{N_1} Q_{L_1}u  , \quad v_{N_2, L_2}=P_{N_2} Q_{L_2} v .$$

We now prove \eqref{keybi-est}--\eqref{CN-biest} by estimating $\mathcal J(N)$ case by case.

\subsection{Case $N_3 \sim N_2 \gg N_1>1$ }
By definition of $X$, decomposition in modulation, Corollary \ref{cor-dydbiest} with $C(N,L)$  as in Proposition \ref{prop-dydbiest} 
we have
\begin{align*}
 \mathcal J(N)
 &\lesssim N_2 \sum_{L_3 \geq 
1}L_3^{-\frac12} \norm{P_{N_3}Q_{L_3} \left(  u_{N_1}   v_{N_2}    \right)}
\\
& \lesssim N_2 \sum_{ L_1, L_2, L_3 \geq 
1}L_3^{-\frac12} \norm{P_{N_3}Q_{L_3} \left(  u_{N_1, L_1} \cdot v_{N_2, L_2}\right)},
\\
& \lesssim N_2 \sum_{ L_1, L_2, L_3 \geq 
1}L_3^{-\frac12}  C(N,L) \|    u_{N_1, L_1} \|_{L^2} \|   v_{N_2, L_2} \|.
\end{align*}

By assumption, \eqref{lmaxmed}, we have
$$
L_{\text{max}}\gtrsim N_1 N^2_2.
$$
If $L_\text{max} \gtrsim N^6_2$,  then we choose $C(N,L)$ as in Proposition \ref{prop-dydbiest}\eqref{cnl3} to obtain
\begin{align*}
\mathcal J(N)
& \lesssim N_1^{\frac12} N_2 \sum_{ L_\text{max} \gtrsim N^6_2}L_3^{-\frac12} L_{\text{min}}^\frac12 (L_1L_2)^{-\frac12}   \left(L_1^\frac12\|   u_{N_1, L_1}\| \right) \left( L_2^\frac12\|   v_{N_2, L_2}\|\right)
\\
& \lesssim  N_2^{-\frac32} 
\norm {  u_{N_1} }_{X} \norm {  v_{N_2} }_{X}.
\end{align*}

Next assume  $L_\text{max} \ll N^6_2$.  Choosing $C(N,L)$ as in Proposition \ref{prop-dydbiest}\eqref{cnl2}, i.e.,
$$
C(N,L)\lesssim (N_1N_2)^{-\frac12} (L_\text{min}L_\text{med})^\frac12, 
$$
we obtain
\begin{align*}
 \mathcal J(N)
& \lesssim N_1^{-\frac12} N_2^\frac12 \sum_{N_1 N^2_2\lesssim  L_\text{max} \ll N^6_2} C(L) \left(L_1^\frac12\|   u_{N_1, L_1}\| \right) \left( L_2^\frac12\|   v_{N_2, L_2}\|\right),
\end{align*}
where
$$
C(L)=L_3^{-\frac12} (L_\text{min}L_\text{med})^{\frac12}(L_1L_2)^{-\frac12}=L_\text{max}^{-\frac12}.
$$
 Now if $L_{\text{max}}\sim L_3$
we have \begin{align*}
\mathcal J(N)
& \lesssim N_1^{-1} N_2^{-\frac12}\norm {  u_{N_1} }_{X} \norm {  v_{N_2} }_{X}.
\end{align*}
If $L_{\text{max}}\sim L_1$ or $L_2$, then
 \begin{align*}
\mathcal J(N)
& \lesssim N_1^{-1} N_2^{-\frac12+}  \norm {  u_{N_1} }_{X} \norm {  v_{N_2} }_{X}.
\end{align*}

\subsection{Case $N_3 \sim N_2 \sim N_1\gg 1$ }
 Proceeding as above, for
 $C(N,L)$
is as in Proposition \ref{prop-dydbiest}, we obtain.
\begin{align*}
\mathcal J(N)& \lesssim N_1 \sum_{ L_1, L_2, L_3 \geq 
1}L_3^{-\frac12}  C(N,L)  \|    u_{N_1, L_1} \| \|   v_{N_2, L_2} \|,
\end{align*}

\subsubsection{Sub-case: $L_{\text{max}} \sim N_1^3 $}
Choosing $C(N,L)$ as in Proposition \ref{prop-dydbiest}\eqref{cnl1}, we get
\begin{align*}
\mathcal J(N)
& \lesssim N_1^\frac34\sum_{ L_{\text{max}} \sim N_1^3 } C(L) \left(L_1^\frac12\|   u_{N_1, L_1}\| \right) \left( L_2^\frac12\|   v_{N_2, L_2}\|\right),
\end{align*}
where
$$
C(L)=L_3^{-\frac12} L_\text{min}^{\frac12}L_\text{med}^{\frac14}(L_1L_2)^{-\frac12}=L_\text{med}^{-\frac14}L_\text{max}^{-\frac12}.
$$
By symmetry, we may assume $L_1\ge L_2$. It suffices to consider the case $L_2\ge L_3$ (the other cases are easier to deal with).
Then 
\begin{align*}
\mathcal J(N)
& \lesssim N_1^\frac34 \sum_{ L_1 \sim N_1^3 }  L_1^{-\frac12}L_2^{-\frac14}\left(L_1^\frac12\|   u_{N_1, L_1}\| \right) \left( L_2^\frac12\|   v_{N_2, L_2}\|\right)
\\
& \lesssim N_1^{-\frac34}  \norm {  u_{N_1} }_{X} 
\sum_{ L_2, L_3: L_2\ge L_3 } L_2^{-\frac14}L_2^{\frac12}\|    v_{N_2, L_2}\|
\\
& \lesssim N_1^{-\frac34}  \norm {  u_{N_1}  }_{X} \norm {  v_{N_2}  }_{X} .
\end{align*}

\subsubsection{Sub-case: $L_{\text{max}} \sim L_{\text{med}}  \gg N_1^3 $}
Choosing $C(N,L)$ as in Proposition \ref{prop-dydbiest}\eqref{cnl3}, we have
\begin{align*}
 \mathcal J(N)
& \lesssim N_1^\frac32 \sum_{ L_{\text{max}} \sim L_{\text{med}} \gg N_1^3 } L_{\text{min }}^\frac12 (L_1L_2)^{-\frac12}\left(L_1^\frac12\|   u_{N_1, L_1}\| \right) \left( L_2^\frac12\|   v_{N_2, L_2}\|\right).
\end{align*}
By symmetry we may assume $L_1\le L_2 \le L_3$ which in turn implies 
$L_2\sim L_3\gg N_1^3$. Then
\begin{align*}
 \mathcal J(N)
& \lesssim N_1^\frac32 \sum_{ L_2 \sim L_3 \gg N_1^3 } L_2^{-\frac12} \left(L_1^\frac12\|   u_{N_1, L_1}\| \right) \left( L_2^\frac12\|   v_{N_2, L_2}\|\right)
\\
& \lesssim N_1^{-\frac34}  \norm {  u_{N_1}  }_{X} \norm {  v_{N_2}  }_{X} .
\end{align*}

\subsection{Case $N_1 \sim N_2 \gg N_3= 1$ }

By definition of $X$, decomposing in modulation and in the output frequency, and using 
 Proposition \ref{prop-dydbiest} we obtain
\begin{align*}
 \mathcal J(N)
 &\lesssim  \sum_{L_3 \geq 
1}L_3^{-\frac12} \norm{P_1 Q_{L_3} \partial_x \left ( u_{N_1} v_{N_2}  \right)}
\\
& \lesssim \sum_{0<M \le 1} \sum_{ L_1, L_2, L_3 \geq 
1}L_3^{-\frac12} M  \norm{ \dot P_{M}Q_{L_3} \left(  u_{N_1, L_1} \cdot v_{N_2, L_2}\right)}
\\
& \lesssim   \sum_{0<M\le 1} \sum_{ L_1, L_2, L_3 \geq 
1}L_3^{-\frac12} M \cdot  C(N,L) \|    u_{N_1, L_1} \| \|   v_{N_2, L_2} \|.
\end{align*}

We may assume $M \ge N_1^{-2}$, since otherwise the desired estimate follows easily.
\subsubsection{Sub-case: $L_{\text{max}} \sim N_1^2 M $}
We choose $C(N,L)$ as in Proposition \ref{prop-dydbiest}\eqref{cnl2}, i.e., 
$$ C(N,L) \lesssim (N_1M)^{-\frac12} ( L_{\text{min}}L_{\text{med}})^\frac12.$$
We may assume $L_{\text{max}}=L_3$, since the other cases are easier.
Then
\begin{align*}
\mathcal J(N)
& \lesssim   \sum_{N_1^{-2}<M\le 1} 
\sum_{ L_1, L_2 \geq 
1}N_1^{-\frac32} \left(L_1^\frac12\|   u_{N_1, L_1}\| \right) \left( L_2^\frac12\|   v_{N_2, L_2}\|\right)
\\
& \lesssim N_1^{-\frac32+}  \norm {  u_{N_1}  }_{X} \norm {  v_{N_2}  }_{X}.
\end{align*}

\subsubsection{Sub-case: $L_{\text{max}} \sim L_{\text{med}}  \gg N_1^2 M $}
We choose $C(N,L)$ as in Proposition \ref{prop-dydbiest}\eqref{cnl3} we obtain
\begin{align*}
 \mathcal J(N)
& \lesssim   \sum_{N_1^{-2}<M\le 1}  \sum_{ L_{\text{max}} \sim L_{\text{med}} \gg N_1^2 N_3 }  M L_{\text{min }}^\frac12 (L_1L_2L_3)^{-\frac12} \left(L_1^\frac12\|   u_{N_1, L_1}\| \right) \left( L_2^\frac12\|   v_{N_2, L_2}\|\right)
\\
& \lesssim N_1^{-2+}  \norm {  u_{N_1}  }_{X} \norm {  v_{N_2}  }_{X}.
\end{align*}

\subsection{Case $N_1 \sim N_2 \gg N_3>1$ }
Proceeding as in Subsection A.1 we obtain
\begin{align*}
 \mathcal J(N)
& \lesssim N_3 \sum_{L_{\text{max}}\gtrsim N_1^2 N_3}L_3^{-\frac12}  C(N,L) \|   u_{N_1, L_1}\|  \|   v_{N_2, L_2}\| .
\end{align*}

If $L_\text{max} \gtrsim N^6_1$,  then choosing $C(N,L)$ as in Proposition \ref{prop-dydbiest}\eqref{cnl3}, we obtain
\begin{align*}
 \mathcal J(N)
& \lesssim N_3^{\frac32}  \sum_{ L_\text{max} \gtrsim N^6_1}L_3^{-\frac12} L_{\text{min}}^\frac12 (L_1L_2)^{-\frac12} \left(L_1^\frac12\|   u_{N_1, L_1}\| \right) \left( L_2^\frac12\|   v_{N_2, L_2}\|\right)
\\
& \lesssim  N_1^{-3} N_3^{\frac32} \norm {  u_{N_1}  }_{X} \norm {  v_{N_2}  }_{X}.
\end{align*}

Next assume $L_\text{max} \ll N^6_1$. In this case we choose $C(N,L)$ as in Proposition \ref{prop-dydbiest}\eqref{cnl2}, i.e.,
$$
C(N,L) \lesssim (N_1N_3)^{-\frac12} (L_\text{min}L_\text{med})^\frac12,$$
to obtain
\begin{align*}
 \mathcal J(N)
& \lesssim N_1^{-\frac12} N_3^\frac12 \sum_{N_1^2 N_3\lesssim  L_\text{max} \ll N^6_1} C(L) \left(L_1^\frac12\|   u_{N_1, L_1}\| \right) \left( L_2^\frac12\|   v_{N_2, L_2}\|\right)
\end{align*}
where
$$
C_L=L_3^{-\frac12} (L_\text{min}L_\text{med})^{\frac12}(L_1L_2)^{-\frac12}=L_\text{max}^{-\frac12}.
$$
 Now if $L_{\text{max}}\sim L_3$
we have \begin{align*}
 \mathcal J(N)
& \lesssim N_1^{-\frac32}   \norm {  u_{N_1}  }_{X} \norm {  v_{N_2}  }_{X}.
\end{align*}
If $L_{\text{max}}\sim L_1$ or $L_2$, then
 \begin{align*}
 \mathcal J(N)
& \lesssim N_1^{-\frac32+}  \norm {  u_{N_1}  }_{X} \norm {  v_{N_2}  }_{X}.
\end{align*}

By symmetry, we may assume $L_1\ge L_2$. It suffices to consider the case $L_2\ge L_3$ (the other cases are easier to deal with).
Then 
\begin{align*}
 \mathcal J(N)
& \lesssim N_1^{-\frac12} N_3^\frac12 \sum_{N_1^2 N_3\lesssim  L_1 \ll N^6_1}   L_1^{-\frac12}  \left(L_1^\frac12\|   u_{N_1, L_1}\| \right) \left( L_2^\frac12\|   v_{N_2, L_2}\|\right)
\\
& \lesssim N_1^{-\frac34}  \norm {  u_{N_1} }_{X} 
\sum_{ L_2, L_3: L_2\ge L_3 } L_2^{-\frac12}L_2^{\frac12}\|   v_{N_2,L_2} \|
\\
& \lesssim N_1^{-\frac34}  \norm {  u_{N_1}  }_{X} \norm {  v_{N_2}  }_{X}.
\end{align*}


\bibliographystyle{amsplain}
\bibliography{kdv}

\end{document}